\documentclass[reqno]{amsart}
\usepackage{amssymb}

\newtheorem{theorem}{Theorem}
\newtheorem{lemma}[theorem]{Lemma}

\theoremstyle{definition}
\newtheorem{definition}[theorem]{Definition}
\newtheorem{setting}{Setting}
\newtheorem{example}[theorem]{Example}

\newtheorem{proposition}[theorem]{Proposition}
\newtheorem{corollary}[theorem]{Corollary}

\newtheorem{notation}{Notation}

\theoremstyle{remark}
\newtheorem{remark}[theorem]{Remark}

\begin{document}

\title[Hartogs companions and holomorphic extensions]{Hartogs companions and holomorphic extensions in arbitrary dimension}
\author{Vlad Timofte}
\address{Institute of Mathematics ``Simion Stoilow'' of the Romanian Academy, P.O. Box 1-764, RO-014700 Bucharest, Romania}
\email{vlad.timofte@imar.ro}

\date{September 7, 2020.}

\begin{abstract}
We show that every holomorphic map $f\in\mathcal{H}(\Omega\setminus K)$ ($K\subset\Omega\subset\mathbb{C}^n$, with $K$ compact, $\Omega$ open, and $n\ge2$), has a unique ``\emph{Hartogs companion}'' $\tilde f\in\mathcal{H}(\Omega)$ matching $f$ on an open subset $C_{K,\Omega}\subset\Omega\setminus K$. Furthermore, $\tilde f$ extends $f$, \emph{if and only if} $\mathbb{C}^n\setminus K$ is a connected set; this equivalence proves the converse implication from the Hartogs Kugelsatz. The existence of vector-valued Hartogs companions in any dimension yields a Hartogs-type extension theorem for G\^ateaux holomorphic maps $f\in\mathcal{H}_\mathrm{G}(\Omega\setminus K,Y)$ on finitely open sets in arbitrary complex vector spaces. The equivalence is very similar to that for $K\subset\Omega\subset\mathbb{C}^n$ and leads to a corresponding Hartogs Kugelsatz in arbitrary dimension and to extension theorems for five types of holomorphy (G\^ateaux, Mackey/Silva, hypoanalytic, Fr\'echet, locally bounded). We also show that the range $\tilde f(\Omega)$ of a vector-valued Hartogs companion cannot leave a domain of holomorphy containing $f(\Omega\setminus K)$. We establish a boundary principle for maps $f\in\mathcal{H}_\mathrm{G}(\Omega,Y)\cap\mathcal{C}(\overline\Omega,Y)$ on finitely bounded open sets. For $Y=\mathbb{C}$, the principle states that $f\big(\overline\Omega\big)=f(\partial\Omega)$ (hence $\sup_{x\in\Omega}|f(x)|=\sup_{x\in\partial\Omega}|f(x)|$). Several results require a new identity theorem, which yields a maximum norm principle and a ``max-min'' seminorm principle.
\end{abstract}

\maketitle

\section{Introduction}\label{s.introduction}

The famous Hartogs extension theorem is a striking and deep result emphasizing the difference between the theory of holomorphic functions of one and of several complex variables:

\begin{theorem}[Hartogs extension]\label{t.Hartogs}
Let $n\ge2$, an open set $\Omega\subset\mathbb{C}^n$, and a compact subset $K\subset\Omega$. If $\Omega\setminus K$ is connected, then every map $f\in\mathcal{H}(\Omega\setminus K)$ has a unique extension $\tilde f\in\mathcal{H}(\Omega)$.
\end{theorem}

A different version of this theorem is the following:

\begin{theorem}[Hartogs Kugelsatz]\label{t.Kugelsatz}
Let $n\ge2$, an open set $\Omega\subset\mathbb{C}^n$, and a compact subset $K\subset\Omega$. If $\mathbb{C}^n\setminus K$ is connected, then the restriction
\[\rho:\mathcal{H}(\Omega)\rightarrow\mathcal{H}(\Omega\setminus K),\qquad\rho(f)=f|_{\Omega\setminus K},\]
is an isomorphism of $\mathbb{C}$-algebras.
\end{theorem}

The surjectivity of the restriction operator $\rho$ is equivalent to the existence of an extension $\tilde f\in\mathcal{H}(\Omega)$ for every map $f\in\mathcal{H}(\Omega\setminus K)$. According to the above results, such extensions exist if one of the sets $\Omega\setminus K$ and $\mathbb{C}^n\setminus K$ is connected.

In this paper we improve and unify Theorems~\ref{t.Hartogs} and~\ref{t.Kugelsatz} as an equivalence of four statements and we generalize them to spaces of infinite dimension. To accomplish this we prove several facts, not necessarily in the order listed below.\\[1mm]
\textbf{1. }We show that every map $f\in\mathcal{H}(\Omega\setminus K)$ has an extension $\tilde f\in\mathcal{H}(\Omega)$, \emph{if and only if} $\mathbb{C}^n\setminus K$ is connected (Theorem~\ref{t.Hartogs=}). Therefore, the ``\emph{right condition}'' in Theorem~\ref{t.Hartogs} is the connectedness of $\mathbb{C}^n\setminus K$, which is also \emph{equivalent} to the surjectivity of the restriction map $\rho$ from Theorem~\ref{t.Kugelsatz}. It turns out that extending every $f\in\mathcal{H}(\Omega\setminus K)$ only depends on $K$, and not on the surrounding open set $\Omega\supset K$. Furthermore, in Theorem~\ref{t.Hartogs} we may replace compactness by significantly weaker assumptions (see Corollary~\ref{c.Hartogs}) allowing $K$ to be unbounded and not closed, and $K\cup(\mathbb{C}^n\setminus\Omega)$ to be path-connected (that is, $K$ ``breaks the boundary'' of $\Omega$), as in Example~\ref{ex.Hartogs}.\\[1mm]
\textbf{2. }Without the connectedness assumption from Theorem~\ref{t.Hartogs}, every $f\in\mathcal{H}(\Omega\setminus K)$ still has a unique ``Hartogs companion'' $\tilde f\in\mathcal{H}(\Omega)$ matching $f$ on a coincidence set (Theorem~\ref{t.C-companion}). The resulting association $f\mapsto\tilde f$ is a left inverse for the restriction operator $\rho$ defined as in Theorem~\ref{t.Kugelsatz}. The range inclusion $\tilde f(\Omega)\subset f(\Omega\setminus K)$ holds, and if $f$ has an extension from $\mathcal{H}(\Omega)$, that must be $\tilde f$. This leads to a striking compact excision property of $\mathbb{C}$-valued holomorphic maps of several variables: \emph{removing a compact from the domain does not change the range} (Corollary~\ref{c.excision}). Dropping connectedness is essential for the construction of vector-valued Hartogs companions in arbitrary dimension (Theorem~\ref{t.Y-companion}, obtained by a slicing technique with linear varieties of finite dimension\footnote{Intersections of a connected set with linear varieties may be disconnected.}). This construction leads to the Hartogs Kugelsatz equivalence in infinite dimension (Theorem~\ref{t.extension=} and Corollary~\ref{c.extension=}).\\
\textbf{3. }In the process, we also prove a needed identity theorem for G\^ateaux holomorphic maps on polygonally connected $2$-open sets (Theorem~\ref{t.identity}). As byproducts we get a maximum norm principle (Theorem~\ref{t.maximum}) and a surprising ``max-min'' seminorm principle (Theorem~\ref{t.max-min}). For $f\in\mathcal{H}_\mathrm{G}(\Omega,Y)$, the latter states that if $p\circ f$ has a local maximum value $M>0$ (the ``max'' assumption) for some continuous seminorm $p$ on $Y$, then $p\circ f\ge M$ everywhere (the ``min'' conclusion) and even though $p\circ f$ may not be constant, the map $f$ vanishes nowhere and its range has empty interior.\\[1mm]
\textbf{4. }Hartogs companions in arbitrary dimension lead to several extension theorems for $K\subset\Omega\subset X$ and G\^ateaux holomorphic maps $f\in\mathcal{H}_\mathrm{G}(\Omega\setminus K,Y)$. We assume at most\footnote{In all three theorems listed here, at least one of the three assumptions is weakened.} that $\Omega$ is finitely open, $K$ is finitely compact, and $\Omega\setminus K$ is polygonally connected (Theorems~\ref{t.extension=}, \ref{t.2-extension}, \ref{t.outer}, and the four corollaries from the last section). Even for $X=\mathbb{C}^n$ and $Y=\mathbb{C}$ some of these results are more general than Theorem~\ref{t.Hartogs}, since $K$ may not be closed or bounded and $\Omega$ may not be open (Example~\ref{ex.d-topology}(c)), while G\^ateaux holomorphy still can be considered on $1$-open sets $\Omega\subset\mathbb{C}^n$.\\[1mm]
\textbf{5. }Several regularity properties (local boundedness, continuity, hypocontinuity, holomorphy, hypoanalyticity, Mackey holomorphy) are inherited from a map by its Hartogs companion (Theorem~\ref{t.regularity}), hence also by G\^ateaux holomorphic extensions whenever these exist. We thus get Hartogs-type extension results for five different types of holomorphy (Corollary~\ref{c.extension=}).\\[1mm]
\textbf{6. }Every domain of holomorphy containing the range of a map $f\in\mathcal{H}_\mathrm{G}(\Omega\setminus K,Y)$, also contains the range of the Hartogs companion $\tilde f$; we call this \emph{range inertia}. In particular, Hartogs-type extensions have this property.\\[1mm]
\textbf{7. } Viewing every map $f\in\mathcal{H}_\mathrm{G}(\Omega,Y)$ as the Hartogs companion of its restrictions of the form $f|_{\Omega\setminus K}$ leads to general boundary principles (Theorem~\ref{t.boundary} and Corollary~\ref{c.boundary}); according to the latter, $f(\overline\Omega)=f(\partial\Omega)$ for every $f\in\mathcal{H}_\mathrm{G}(\Omega)\cap\mathcal{C}(\overline\Omega)$, where $\Omega$ is a $2$-bounded open set in a Hausdorff topological vector space.

Due to its very interesting properties and consequences, we may conclude that the Hartogs companion is a new flexible tool which deserves further investigation.

\section{Hartogs companions in finite dimension}\label{s.finite}

\subsection{Hartogs companions in dimension at least two.}\label{ss.ge2}

For arbitrary complex vector space $X$, sets $A,B\subset X$ and $S\subset\mathbb{C}$, and elements $u\in X$, $\lambda\in\mathbb{C}$, it is convenient to write
\begin{eqnarray*}
&&A+B:=\{a+b\,|\,a\in A,\,b\in B\},\qquad S\cdot A:=\{sa\,|\,s\in S,\,a\in A\},\\
&&A+u=u+A:=A+\{u\},\qquad S\cdot u:=S\cdot\{u\},\qquad\lambda\cdot A:=\{\lambda\}\cdot A.
\end{eqnarray*}

\begin{setting}\label{set.Hartogs}
\emph{Throughout Section~\ref{ss.ge2}, for arbitrary integer $n\ge2$ we consider an open set $\Omega\subset\mathbb{C}^n$ and a compact subset $K\subset\Omega$.}
\end{setting}

For shortness, the connected components of any subset of $\mathbb{C}^n$ will be simply called components. We denote by $\Upsilon_\Omega$ the set of all components of $\Omega$.

By removing connectedness assumptions, the following result accomplishes the construction of Hartogs companions in finite dimension. Furthermore, its last part yields both Theorems~\ref{t.Hartogs} and~\ref{t.Kugelsatz}. As Theorem~\ref{t.Hartogs=} will show, \emph{the Hartogs phenomenon for holomorphic functions is characterized by the connectedness of $\mathbb{C}^n\setminus K$}.

\begin{theorem}[Hartogs companions in finite dimension]\label{t.C-companion}
Let us define the coincidence (open) set of the inclusion $K\subset\Omega$ as
\[C_{K,\Omega}:=\bigcup_{\omega\in\Upsilon_\Omega}(\omega\cap K_\omega^\mathrm{u})\subset\Omega\setminus K,\]
where $K_\omega^\mathrm{u}$ denotes the unbounded component of $\mathbb{C}^n\setminus(K\cap\omega)$. Let an arbitrary map $f\in\mathcal{H}(\Omega\setminus K)$. Then
\begin{description}
\item[(a)] There exists a unique map $\tilde f\in\mathcal{H}(\Omega)$ (which will be called the Hartogs companion of $f$), such that
\begin{equation}\label{e.companion}\tilde f|_{C_{K,\Omega}}=f|_{C_{K,\Omega}}.\end{equation}
Furthermore, we have the range inclusion
\begin{equation}\label{e.range}\tilde f(\Omega)\subset f(\Omega\setminus K).\end{equation}
If $f$ has an extension $\bar f\in\mathcal{H}(\Omega)$, then $\bar f=\tilde f$.
\item[(b)] If $K\subset K_0\subset\Omega_0\subset\Omega$, with $K_0$ compact and $\Omega_0$ open, then $\tilde f|_{\Omega_0}$ is the Hartogs companion of $f|_{\Omega_0\setminus K_0}$.
\item[(c)] For arbitrarily fixed $a\in\Omega$ and $u\in\mathbb{C}^n\setminus\{0\}$, let a bounded open set $G\subset\mathbb{C}$ with the boundary $\partial G$ consisting of finitely many piecewise $\mathcal{C}^1$ Jordan curves, and satisfying the following condition denoted by $C_{G,u}(a)$\textup{:}
\[0\in G,\qquad a+\overline G\cdot u\subset\Omega,\qquad K\cap(a+\mathbb{C}\cdot{u})\subset a+G\cdot u\]
(any set $G$ with the above properties will be called $(a,u)$-admissible\footnote{Or more specific, whenever needed: $(a,u)$-admissible for the inclusion $K\subset\Omega$.}).
Then $\Omega_{G,u}:=\{x\in\Omega\,|\,C_{G,u}(x)\mbox{ holds}\}$ is an open neighborhood of $a$. We have the inclusion $\Omega_{G,u}+\partial G\cdot u\subset\Omega\setminus K$ and the representation formula (where the boundary $\partial G$ is oriented such that $G$ lies to the left of $\partial G$)
\begin{equation}\label{e.representation}\tilde f(x)=\frac1{2\pi\mathrm{i}}\int_{\partial G}\frac{f(x+\zeta u)}\zeta\mathrm{d}\zeta,\quad\mbox{for every }x\in\Omega_{G,u}.\end{equation}
Hence for every linear variety $L\subset\mathbb{C}^n$ of dimension at least one, such that $\Omega_L:=\Omega\cap L\ne\emptyset$, the restriction $\tilde f|_{\Omega_L}$ is uniquely determined by $f|_{\Omega_L\setminus K}$.
\item[(d)] If $\mathbb{C}^n\setminus K$ is connected, then $C_{K,\Omega}=\Omega\setminus K$, and hence $\tilde f|_{\Omega\setminus K}=f$.
\end{description}
\end{theorem}

\begin{proof}
To shorten notation, we write $C_{K,\Omega}$ and $\Upsilon_\Omega$ simply as $C$ and $\Upsilon$, respectively. There is no loss of generality in assuming $K\ne\emptyset$.\\
(a). The uniqueness of $\tilde f$ follows easily by (\ref{e.companion}) and the identity theorem applied to $\tilde f|_\omega$ for each component $\omega\in\Upsilon$, since $\omega\cap C=\omega\cap K_\omega^\mathrm{u}\ne\emptyset$ is an open subset of $\omega$. Therefore, we only need to prove the existence part for an arbitrarily fixed $\omega\in\Upsilon$, its compact subset $K_\omega:=K\cap\omega$, and the restriction $f|_{\omega\setminus K_\omega}\in\mathcal{H}(\omega\setminus K_\omega)$. Consequently, there is no loss of generality in assuming that $\Omega$ is connected. Thus $C=\Omega\cap K^\mathrm{u}$, where $K^\mathrm{u}$ denotes the unbounded component of $\mathbb{C}^n\setminus K$. Let us fix $\chi\in\mathcal{C}_0^\infty(\Omega)$, with $\chi\equiv1$ on some neighborhood of $K$ (throughout the proof, any such $\chi$ will be called a $K$-map). Set $K_\chi:=\mathrm{supp}\,\chi\subset\Omega$.\\
\emph{The construction below of the map $\tilde f_\chi\in\mathcal{H}(\Omega)$ follows the idea from the proof\footnote{The cited proof uses the $\bar\partial$ technique initiated by Ehrenpreis\,\cite{ehrenpreis}.} of Theorem~2.3.2 from H\"ormander\,\cite{hormander} (p.30)}. There is a (unique) map $f_\chi\in\mathcal{C}^\infty(\Omega)$, such that $f_\chi=(1-\chi)f$ on $\Omega\setminus K$ and $f_\chi|_K\equiv0$. Hence $f_\chi=f$ on $\Omega\setminus K_\chi$. Since $f\in\mathcal{H}(\Omega\setminus K)$, for the smooth complex differential $(0,1)$-form on $\Omega$ defined by
\[\bar\partial f_\chi=\sum_{j=1}^n\frac{\partial f_\chi}{\partial\bar z_j}\mathrm{d}\bar z_j,\]
we have $\mathrm{supp}\,\bar\partial f_\chi\subset K_\chi$. Therefore, $\bar\partial f_\chi$ extends (by $0$) to a differential $(0,1)$-form $g=\sum_{j=1}^ng_j\mathrm{d}\bar z_j$, with $g_1,\dots,g_n\in\mathcal{C}_0^\infty(\mathbb{C}^n)$ and $\mathrm{supp}\,g\subset K_\chi$. It follows that $\bar\partial g|_\Omega=\bar\partial^2f_\chi\equiv0$ and $\bar\partial g|_{\mathbb{C}^n\setminus K_\chi}\equiv0$, that is, $\bar\partial g\equiv0$. For the map $h\in\mathcal{C}_0^\infty(\mathbb{C}^n)$ defined by (see Th.\,2.3.1 from H\"ormander\,\cite{hormander} and its proof)
\begin{equation}\label{e.h}h(z_1,\dots,z_n)=\frac1{2\pi\mathrm{i}}\int_\mathbb{C}\frac{g_1(\zeta,z_2,\dots,z_n)}{\zeta-z_1}\mathrm{d}\zeta\wedge\mathrm{d}\bar\zeta,\end{equation}
we have $\bar\partial h=g$. Since $K\subset K_\chi\subset\Omega$, for the unbounded component $K_\chi^\mathrm{u}$ of $\mathbb{C}^n\setminus K_\chi$ it is easily seen that
\[K_\chi^\mathrm{u}\subset K^\mathrm{u},\qquad\Omega\cap K_\chi^\mathrm{u}\ne\emptyset.\]
On $K_\chi^\mathrm{u}\subset\mathbb{C}^n\setminus\mathrm{supp}\,g$ we have $\bar\partial h=g\equiv0$, and so $h|_{K_\chi^{^\mathrm{u}}}\in\mathcal{H}(K_\chi^\mathrm{u})$. As $h\in\mathcal{C}_0^\infty(\mathbb{C}^n)$ vanishes on the open set $K_\chi^\mathrm{u}\setminus\mathrm{supp}\,h\ne\emptyset$, by the identity theorem we get $h|_{K_\chi^{^\mathrm{u}}}\equiv0$. For $\tilde f_\chi:=f_\chi-h|_\Omega\in\mathcal{C}^\infty(\Omega)$, we have $\bar\partial\tilde f_\chi=\bar\partial f_\chi-g|_\Omega\equiv0$, that is, $\tilde f_\chi\in\mathcal{H}(\Omega)$. Clearly, $\tilde f_\chi|_{\Omega\cap K_\chi^{^\mathrm{u}}}=f_\chi|_{\Omega\cap K_\chi^{^\mathrm{u}}}-h|_{\Omega\cap K_\chi^{^\mathrm{u}}}=f|_{\Omega\cap K_\chi^{^\mathrm{u}}}$. Let us observe that the map $\tilde f:=\tilde f_\chi\in\mathcal{H}(\Omega)$ does not depend on the choice of $\chi$. Indeed, for every $K$-map $\eta$ with $K_\eta:=\mathrm{supp}\,\eta\subset K_\chi$, we have $K_\chi^\mathrm{u}\subset K_\eta^\mathrm{u}$. This yields $\tilde f_\chi|_{\Omega\cap K_\chi^{^\mathrm{u}}}=f|_{\Omega\cap K_\chi^{^\mathrm{u}}}=\tilde f_\eta|_{\Omega\cap K_\chi^{^\mathrm{u}}}$, which forces $\tilde f_\chi=\tilde f_\eta$, by the identity theorem. Hence $\tilde f$ has the property that
\begin{equation}\label{e.Kru}\tilde f|_{\Omega\cap K_\chi^{^\mathrm{u}}}=f|_{\Omega\cap K_\chi^{^\mathrm{u}}},\quad\mbox{for every $K$-map }\chi.\end{equation}
In order to show that $\tilde f|_C=f|_C$, let us fix $a\in C=\Omega\cap K^\mathrm{u}$. Choose $c\in\mathbb{C}^n$, with $\|c\|>\max_{x\in K}\|x\|$. Thus $\delta:=[1,\infty[\cdot c\subset K^\mathrm{u}$. As $K^\mathrm{u}$ is a connected open set, there is a polygonal chain $\Lambda\subset K^\mathrm{u}$ joining $a$ to $c$. Hence $\Delta:=\delta\cup\Lambda\subset K^\mathrm{u}$. There exists a $K$-map $\chi$, such that $K_\chi\cap\Delta=\emptyset$. Since $\Delta$ is an unbounded connected set, it follows that $\Delta\subset K_\chi^\mathrm{u}$, and so $a\in\Omega\cap K_\chi^\mathrm{u}$. By (\ref{e.Kru}) we deduce that $\tilde f(a)=f(a)$. As $a$ was arbitrary, we conclude that $\tilde f|_C=f|_C$. We thus have proved the existence and uniqueness of the map $\tilde f\in\mathcal{H}(\Omega)$ satisfying (\ref{e.companion}). Let us observe that whenever $f$ has an extension $\bar f\in\mathcal{H}(\Omega)$, then $\bar f|_C=f|_C$ yields $\bar f=\tilde f$, by the uniqueness of $\tilde f$.\\
\emph{The range inclusion} (\ref{e.range}). Let us fix $z\in\mathbb{C}\setminus f(\Omega\setminus K)$. For $f_z,h,g\in\mathcal{H}(\Omega\setminus K)$ and $p\in\mathcal{H}(\Omega)$, defined by $f_z:=f-z$, $h:=\frac1{f_z}$, $g:=f_zh$, $p\equiv1$, we have $(\tilde f_z\tilde h)|_C=g|_C$ and $(\tilde f-z)|_C=f_z|_C$ and $p|_C=g|_C$. By the uniqueness of the Hartogs companions we get $\tilde g=\tilde f_z\tilde h$ and $\tilde f_z=\tilde f-z$ and $\tilde g=p$. It follows that $(\tilde f-z)\tilde h=\tilde f_z\tilde h=\tilde g=p\equiv1$, which yields $z\in\mathbb{C}\setminus\tilde f(\Omega)$. We thus conclude that $\tilde f(\Omega)\subset f(\Omega\setminus K)$. It $f$ has an extension from $\mathcal{H}(\Omega)$, that must be $\tilde f$. In this case, $f(\Omega\setminus K)=\tilde f(\Omega\setminus K)\subset\tilde f(\Omega)$, and so $\tilde f(\Omega)=f(\Omega\setminus K)$. For another proof of the range inclusion (\ref{e.range}), see Proposition~\ref{p.range} (independent of any preceeding results) and the comment at the end of its proof.\\
(c). Let us consider $a\in\Omega$ and $u\in\mathbb{C}^n\setminus\{0\}$, and an $(a,u)$-admissible set $G\subset\mathbb{C}$. Clearly, such sets $G$ exist (in $\mathbb{C}\cdot u\simeq\mathbb{C}$, the compact $(K-a)\cap(\mathbb{C}\cdot u)$ may be covered by a finite union of open balls with the closures contained in the open set $(\Omega-a)\cap(\mathbb{C}\cdot u)$). For every $x\in\mathbb{C}^n$, the last two conditions from $C_{G,u}(x)$ may be written as (where $G^\mathrm{c}:=\mathbb{C}\setminus G$ and $A^\mathrm{c}:=\mathbb{C}^n\setminus A$ for every $A\subset\mathbb{C}^n$)
\begin{eqnarray*}
x+\overline G\cdot u\subset\Omega\!\!\!\!&\iff&\!\!\!\!\big(x+\overline G\cdot u\big)\cap\Omega^\mathrm{c}=\emptyset\iff x\notin\Omega^\mathrm{c}-\overline G\cdot u,\\
K\cap(x+\mathbb{C}\cdot u)\subset x+G\cdot u\!\!\!\!&\iff&\!\!\!\!K\cap(x+G^\mathrm{c}\cdot u)=\emptyset\iff x\notin K-G^\mathrm{c}\cdot u,
\end{eqnarray*}
and these obviously yield $x\in\Omega$ and $x+\partial G\cdot u\subset\Omega\setminus K$. We thus get
\[\Omega_{G,u}=\big(\Omega^\mathrm{c}-\overline G\cdot u\big)^\mathrm{c}\cap\big(K-G^\mathrm{c}\cdot u\big)^\mathrm{c}\subset\Omega,\qquad\Omega_{G,u}+\partial G\cdot u\subset\Omega\setminus K.\]
All four sets $\overline G\cdot u$, $K$, $\Omega^\mathrm{c}$, $G^\mathrm{c}\cdot u$, are closed in $\mathbb{C}^n$, and the first two are compact. Hence both $\Omega^\mathrm{c}-\overline G\cdot u$ and $K-G^\mathrm{c}\cdot u$ are closed, and so $\Omega_{G,u}$ is open. Clearly, $a\in\Omega_{G,u}$. In order to prove (\ref{e.representation}), let us fix $x\in\Omega_{G,u}$. By using a linear change of coordinates in $\mathbb{C}^n$, we may assume $u=(1,0,\dots,0)$. Set $x^1:=(x_2,\dots,x_n)\in\mathbb{C}^{n-1}$ and $L:=x+\mathbb{C}\cdot u=\mathbb{C}\times\{x^1\}$. Thus $\Omega_L=\{z_1\in\mathbb{C}\,|\,(z_1,x^1)\in\Omega\}$. Let us choose another $(x,u)$-admissible set $G_0\subset\mathbb{C}$, such that $\overline G_0\subset G$. Hence both conditions $C_{G,u}(x)$ and $C_{G_0,u}(x)$ hold, and so
\[0\in G_0,\qquad x+\overline G\cdot u\subset\Omega_L,\qquad K\cap L\subset x+G_0\cdot u.\]
Since $K_0:=K\cup(x+\overline G_0\cdot u)$ is compact, $\Omega_0:=\Omega\setminus(x+G^\mathrm{c}\cdot u)$ is open, and $K_0\subset\Omega_0$, there exists $\chi\in\mathcal{C}_0^\infty(\Omega)$, such that $\chi\equiv1$ on some neighborhood $\Omega_1\subset\Omega_0$ of $K_0\supset K$ and $K_\chi=\mathrm{supp}\,\chi\subset\Omega_0$. Let us observe that
\[K_0\cap L=x+\overline G_0\cdot u,\qquad K_\chi\cap L\subset x+G\cdot u.\]
By the construction from (a) it follows that $\tilde f=f_\chi-h|_\Omega$, where $h$ is defined by (\ref{e.h}) and $g_1\in\mathcal{C}_0^\infty(\mathbb{C}^n)$ is the extension by $0$ of $\frac{\partial f_\chi}{\partial\bar z_1}\in\mathcal{C}_0^\infty(\Omega)$. Hence $\mathrm{supp}\,g_1\subset K_\chi\setminus\Omega_1$. For $S:=G\setminus\overline G_0\subset\mathbb{C}$, we have $\overline S=\overline G\setminus G_0$ and $\partial S=\partial G\cup\partial G_0$, and so
\begin{eqnarray*}
x+\overline S\cdot u\!\!\!\!&=&\!\!\!\!(x+\overline G\cdot u)\setminus(x+G_0\cdot u)\subset(x+\overline G\cdot u)\setminus(K\cap L)\subset\Omega\setminus K,\\
\mathrm{supp}\,(g_1|_{\Omega_L})\!\!\!\!&\subset&\!\!\!\!\mathrm{supp}\,g_1\cap L\subset(K_\chi\setminus\Omega_1)\cap L\subset(K_\chi\cap L)\setminus(K_0\cap L)\\
&\subset&\!\!\!\!(x+G\cdot u)\setminus(x+\overline G_0\cdot u)=x+S\cdot u=(x_1+S)\times\{x^1\}.
\end{eqnarray*}
Hence $\mathrm{supp}\,[g_1(\cdot,x^1)]\subset x_1+S$. Since $f_\chi\equiv0$ on $\Omega_1\supset K_0\supset x+\partial G_0\cdot u$ and $f_\chi=f$ on $\Omega\setminus K_\chi\supset\Omega_L\setminus(K_\chi\cap L)\supset x+\partial G\cdot u$, by the definitions of $\tilde f,f_\chi,g_1,h$, together with (\ref{e.h}) and Green's formula (for the $\mathcal{C}^\infty$ differential $1$-form $\zeta\mapsto\frac{f_\chi(\zeta+x_1,x^1)}\zeta\mathrm{d}\zeta$ on $\mathbb{C}\setminus\{0\}\supset\overline S$ and the compact set $\overline S$ with piecewise $\mathcal{C}^1$ boundary), it follows that
\begin{eqnarray*}
\tilde f(x)\!\!\!\!&=&\!\!\!\!-h(x)=\frac1{2\pi\mathrm{i}}\int_\mathbb{C}\frac{g_1(\zeta,x^1)}{\zeta-x_1}\mathrm{d}\bar\zeta\wedge\mathrm{d}\zeta=\frac1{2\pi\mathrm{i}}\int_{x_1+S}\frac{\frac{\partial f_\chi}{\partial\bar z_1}(\zeta,x^1)}{\zeta-x_1}\mathrm{d}\bar\zeta\wedge\mathrm{d}\zeta\\
&=&\!\!\!\!\frac1{2\pi\mathrm{i}}\int_S\frac{\frac{\partial f_\chi}{\partial\bar z_1}(\zeta+x_1,x^1)}\zeta\mathrm{d}\bar\zeta\wedge\mathrm{d}\zeta=\frac1{2\pi\mathrm{i}}\int_S\frac\partial{\partial\bar\zeta}\left(\frac{f_\chi(\zeta+x_1,x^1)}\zeta\right)\mathrm{d}\bar\zeta\wedge\mathrm{d}\zeta\\
&=&\!\!\!\!\frac1{2\pi\mathrm{i}}\int_{\partial S}\frac{f_\chi(\zeta+x_1,x^1)}\zeta\mathrm{d}\zeta=\frac1{2\pi\mathrm{i}}\int_{\partial S}\frac{f_\chi(x+\zeta u)}\zeta\mathrm{d}\zeta=\frac1{2\pi\mathrm{i}}\int_{\partial G}\frac{f(x+\zeta u)}\zeta\mathrm{d}\zeta.
\end{eqnarray*}
We thus have proved the representation formula (\ref{e.representation}). Now let us fix a linear variety $L\subset\mathbb{C}^n$ as in (c). For $a\in\Omega_L$ and $u\in(L-a)\setminus\{0\}$, we may write $\tilde f(a)$ as in (\ref{e.representation}). Since $a+\mathbb{C}\cdot u\subset L$, we conclude that $\tilde f(a)$ only depends on the restriction $f|_{\Omega_L\setminus K}$.\\
(b). According to (a), the map $f_0:=f|_{\Omega_0\setminus K_0}$ has a unique Hartogs companion $\tilde f_0\in\mathcal{H}(\Omega_0)$. For fixed $a\in\Omega_0$, let us choose $u\in\mathbb{C}^n\setminus\{0\}$ and a set $G\subset\mathbb{C}$, which is $(a,u)$-admissible for the inclusion $K_0\subset\Omega_0$, and hence also for $K\subset\Omega$. By using the representation formula (\ref{e.representation}) for both $\tilde f$ and $\tilde f_0$ it follows that $\tilde f(a)=\tilde f_0(a)$. As $a$ was arbitrary, we conclude that $\tilde f|_{\Omega_0}=\tilde f_0$.\\
(d). Assume $K^\mathrm{c}=\mathbb{C}^n\setminus K$ is connected. We claim that for every $\omega\in\Upsilon$, the set $\mathbb{C}^n\setminus K_\omega$ is connected. To show this, let us fix $\omega\in\Upsilon$. There is no restriction in assuming $\Upsilon\ne\{\omega\}$ (otherwise, $\omega=\Omega$ and $\mathbb{C}^n\setminus K_\omega=K^\mathrm{c}$ is connected). We have
\[\mathbb{C}^n\setminus K_\omega=K^\mathrm{c}\cup\omega^\mathrm{c}=K^\mathrm{c}\cup\Omega^\mathrm{c}\cup(\Omega\setminus\omega)=K^\mathrm{c}\cup(\Omega\setminus\omega)=\bigcup_{\omega'\in\Upsilon\setminus\{\omega\}}(K^\mathrm{c}\cup\omega').\]
For every $\omega'\in\Upsilon$, the set $K^\mathrm{c}\cup\omega'$ is connected, since so are both $K^\mathrm{c}$ and $\omega'$, and $K^\mathrm{c}\cap\omega'=\omega'\setminus K\ne\emptyset$. Since $\bigcap_{\omega'\in\Upsilon\setminus\{\omega\}}(K^\mathrm{c}\cup\omega')\supset K^\mathrm{c}\ne\emptyset$, the set $\mathbb{C}^n\setminus K_\omega$ is connected. Our claim is proved. Hence $\omega\cap K_\omega^\mathrm{u}=\omega\setminus K_\omega=\omega\setminus K$. As $\omega$ was arbitrary, we get $C=\bigcup_{\omega\in\Upsilon}(\omega\setminus K)=\Omega\setminus K$, which yields $\tilde f|_{\Omega\setminus K}=f$, by (\ref{e.companion}).
\end{proof}

\begin{remark}[coincidence set]\label{r.companion}
Since $\Omega\cap K^\mathrm{u}\subset C_{K,\Omega}$, the map $\tilde f$ from Theorem~\ref{t.C-companion}(a) also satisfies the simpler, but weaker condition\footnote{The conditions (\ref{e.companion}) and (\ref{e.weak}) coincide if $\Omega$ is connected.}
\begin{equation}\label{e.weak}\tilde f|_{\Omega\cap K^{^\mathrm{u}}}=f|_{\Omega\cap K^{^\mathrm{u}}},\end{equation}
where $K^\mathrm{u}$ denotes the unbounded component of $\mathbb{C}^n\setminus K$. The above condition yields the uniqueness of $\tilde f$, if and only if $\omega\cap K^\mathrm{u}\ne\emptyset$ for every component $\omega\in\Upsilon_\Omega$. For instance, with the notation $B_r:=B_{\mathbb{C}^n}(0,r)$ for $r>0$, the open sets $\omega_1:=B_1$ and $\omega_2:=B_4\setminus\overline B_2$ are the components of $\Omega:=\omega_1\cup\omega_2$, which contains the compact set $K:=\partial B_3$. We have $\omega_1\cap K_{\omega_1}^\mathrm{u}=B_1$ and $\omega_2\cap K_{\omega_2}^\mathrm{u}=\Omega\cap K^\mathrm{u}=B_4\setminus\overline B_3$. Hence for arbitrarily given $f\in\mathcal{H}(\Omega\setminus K)$, the condition (\ref{e.weak}) determines $\tilde f|_{\omega_2}$, but not $\tilde f|_{\omega_1}$.
\end{remark}

\begin{remark}[left inverse]\label{r.left}
By Theorem~\ref{t.C-companion}(a), the \emph{Hartogs companion operator}
\[H_\mathbb{C}:\mathcal{H}(\Omega\setminus K)\rightarrow\mathcal{H}(\Omega),\qquad H_\mathbb{C}(f)=\tilde f,\]
is a morphism of $\mathbb{C}$-algebras and a left inverse for the restriction operator $\rho$ defined as in Theorem~\ref{t.Kugelsatz}, that is,
\[\widetilde{\rho(g)}=g,\quad\mbox{for every }g\in\mathcal{H}(\Omega).\]
\end{remark}

\begin{proposition}[composition property of Hartogs companions]\label{p.C-composition}
Let an open set $D\subset\mathbb{C}$ and a map $g\in\mathcal{H}(D)$. Then
\[\widetilde{(g\circ f)}=g\circ\tilde f,\quad\mbox{for every }f\in\mathcal{H}(\Omega\setminus K)\mbox{ with }f(\Omega\setminus K)\subset D.\]
\end{proposition}

\begin{proof}
Let $f\in\mathcal{H}(\Omega\setminus K)$ as in the claimed property. Hence $\tilde f(\Omega)\subset f(\Omega\setminus K)\subset D$ by (\ref{e.range}), and so $g\circ f\in\mathcal{H}(\Omega\setminus K)$ and $g\circ\tilde f\in\mathcal{H}(\Omega)$. Since $\big(g\circ\tilde f\big)|_{C_{K,\Omega}^d}=(g\circ f)|_{C_{K,\Omega}^d}$, by the uniqueness of the Hartogs companion we conclude that $\widetilde{(g\circ f)}=g\circ\tilde f$.
\end{proof}

We next establish some relevant equivalences for the connectedness assumptions used by Theorems~\ref{t.Hartogs}, \ref{t.Kugelsatz} and~\ref{t.C-companion}(d).

\begin{proposition}[connectedness conditions]\label{p.connectedness}
We have the equivalences
\begin{eqnarray}
\label{e.1}\Omega\setminus K\mbox{ is connected}\!\!\!\!&\iff&\!\!\!\!\mathbb{C}^n\setminus K\mbox{ and }\Omega\mbox{ are connected}.\\
\label{e.2}\mathbb{C}^n\setminus K\mbox{ is connected}\!\!\!\!&\iff&\!\!\!\!\mathbb{C}^n\setminus K_\omega\mbox{ is connected for every }\omega\in\Upsilon_\Omega\\
\label{e.3}&\iff&\!\!\!\!\omega\setminus K\mbox{ is connected for every }\omega\in\Upsilon_\Omega.
\end{eqnarray}
\end{proposition}

\begin{proof}
There is no restriction in assuming $K\ne\emptyset$. Throughout this proof the three equivalences from the proposition will be referred to as ``$\stackrel{_{(i)}}\Leftrightarrow$'', with $i\in\{\ref{e.1},\ref{e.2},\ref{e.3}\}$.\\
``$\stackrel{_{(\ref{e.1})}}\Rightarrow$''. Assume $\Omega\setminus K$ is connected. It is easily seen that for all $a\in\Omega$ and $b\in\mathbb{C}^n\setminus K$, there exist $a',b'\in\Omega\setminus K$, such that $[a,a']\subset\Omega$ and $[b,b']\subset\mathbb{C}^n\setminus K$. This yields the connectedness of both sets $\Omega$ and $\mathbb{C}^n\setminus K$, since in $\Omega\setminus K$ any to points can be joined by a polygonal chain.\\
``$\stackrel{_{(\ref{e.1})}}\Leftarrow$''. On the contrary, assume both $\mathbb{C}^n\setminus K$ and $\Omega$ are connected, but $\Omega\setminus K$ is not. Thus $\Omega\setminus K=\Omega_1\cup\Omega_2$, for some disjoint nonempty open sets $\Omega_1,\Omega_2\subset\Omega$. By Theorem~\ref{t.C-companion}(d), the map $f\in\mathcal{H}(\Omega\setminus K)$ defined by $f|_{\Omega_j}\equiv j$ ($j\in\{1,2\}$) has a unique extension $\tilde f\in\mathcal{H}(\Omega)$. For $j\in\{1,2\}$, since $\Omega$ is connected and $\tilde f|_{\Omega_j}=f|_{\Omega_j}\equiv j$, by the identity theorem it follows that $\tilde f\equiv j$. We thus get $1\equiv\tilde f\equiv2$, a contradiction. We conclude that $\Omega\setminus K$ is connected. The first equivalence of the proposition is proved. This also yields ``$\stackrel{_{(\ref{e.3})}}\Leftrightarrow$'', since every $\omega\in\Upsilon_\Omega$ is connected and $\omega\setminus K=\omega\setminus K_\omega$.\\
``$\stackrel{_{(\ref{e.2})}}\Rightarrow$''. If $\mathbb{C}^n\setminus K$ is connected, then so is $\mathbb{C}^n\setminus K_\omega$ for every component $\omega\in\Upsilon_\Omega$, as already shown by the first part (the claim) of the proof of Theorem~\ref{t.C-companion}(d).\\
``$\stackrel{_{(\ref{e.2})}}\Leftarrow$''. Assume $\mathbb{C}^n\setminus K_\omega$ is connected for every $\omega\in\Upsilon_\Omega$. By the third equivalence, all $\omega\setminus K$ are connected. Let us fix $a,b\in\mathbb{C}^n\setminus K$. Let a polygonal chain $\Lambda\subset\mathbb{C}^n$ joining $a$ to $b$. The set $F_\Lambda:=\{\omega\in\Upsilon_\Omega\,|\,\Lambda\cap K_\omega\ne\emptyset\}$ is finite, since the compact $\Lambda\cap K\subset\Omega=\bigcup_{\omega\in\Upsilon_\Omega}\omega$ may be covered by finitely many $\omega\in\Upsilon_\Omega$, which are disjoint. Now let us choose $\Lambda$ for which the cardinality of $F_\Lambda$ is the smallest possible. We claim that $F_\Lambda=\emptyset$. On the contrary, suppose there exists $\omega\in F_\Lambda\ne\emptyset$. There is a path $\gamma:[0,1]\rightarrow\mathbb{C}^n$, such that $\gamma(0)=a$, $\gamma(1)=b$, and $\gamma([0,1])=\Lambda$. As $\Lambda\cap K_\omega$ is compact, $\omega$ is open, and $a,b\in\mathbb{C}^n\setminus K_\omega$, there exist $t,s\in\,]0,1[$, such that $t<s$ and
\[\gamma(t),\gamma(s)\in\omega\setminus K,\qquad\gamma([0,t])\cup\gamma([s,1])\subset\mathbb{C}^n\setminus K_\omega.\]
As $\omega\setminus K$ is connected, there is a polygonal chain $\Lambda_\omega\subset\omega\setminus K$ joining $\gamma(t)$ to $\gamma(s)$. The polygonal chain $\Psi=\gamma([0,t])\cup\Lambda_\omega\cup\gamma([s,1])$ joins $a$ to $b$ and $F_\Psi\subset F_\Lambda\setminus\{\omega\}$, which contradicts the choice of $\Lambda$. Hence $F_\Lambda=\emptyset$, that is, $\Lambda\subset\mathbb{C}^n\setminus K$. We thus conclude that $\mathbb{C}^n\setminus K$ is connected.
\end{proof}

Summarizing, we can now unify Theorems~\ref{t.Hartogs}, \ref{t.Kugelsatz} and~\ref{t.C-companion}(a,d) as follows:

\begin{theorem}[Hartogs extension/Kugelsatz in $\mathbb{C}^n$]\label{t.Hartogs=}
Let $n\ge2$, an open set $\Omega\subset\mathbb{C}^n$, and a compact subset $K\subset\Omega$. The following four statements are equivalent.
\begin{description}
\item[(i)] Every map $f\in\mathcal{H}(\Omega\setminus K)$ has a (unique) extension $\tilde f\in\mathcal{H}(\Omega)$.
\item[(i')] Every locally constant map $g:\Omega\setminus K\rightarrow\mathbb{C}$ has an extension $\tilde g\in\mathcal{H}(\Omega)$.
\item[(ii)] The restriction $\rho:\mathcal{H}(\Omega)\rightarrow\mathcal{H}(\Omega\setminus K)$ is an isomorphism\footnote{The inverse of $\rho$ is the Hartogs companion operator $H_\mathbb{C}$ from Remark~\ref{r.left}.} of $\mathbb{C}$-algebras.
\item[(iii)] $\mathbb{C}^n\setminus K$ is connected.
\end{description}
In this case the isomorphisms $\rho$ and $H_\mathbb{C}$ are range preserving, that is,
\[f(\Omega)=f(\Omega\setminus K),\quad\mbox{for every }f\in\mathcal{H}(\Omega).\]
\end{theorem}

\begin{proof}
Since the implications (iii)$\Rightarrow$(ii)$\Rightarrow$(i)$\Rightarrow$(i') are clear, we only need to prove that (i')$\Rightarrow$(iii). Assume (i') holds, but $\mathbb{C}^n\setminus K$ is disconnected. By Proposition~\ref{p.connectedness}, $\omega\setminus K=\omega_1\cup\omega_2$ for some $\omega\in\Upsilon_\Omega$ and disjoint nonempty open sets $\omega_1,\omega_2\subset\omega\setminus K$. Let a locally constant map $g:\Omega\setminus K\rightarrow\mathbb{C}$, such that $g|_{\omega_j}\equiv j$ for $j\in\{1,2\}$. According to (i'), $g$ has an extension $\tilde g\in\mathcal{H}(\Omega)$. Since $\omega$ is connected, by the identity theorem we get $1\equiv\tilde g|_\omega\equiv2$, a contradiction. Hence $\mathbb{C}^n\setminus K$ is connected. We thus have proved the equivalence of the four statements. By Theorem~\ref{t.C-companion}(a) we see that every $f\in\mathcal{H}(\Omega)$ is the Hartogs companion of its restriction $h:=f|_{\Omega\setminus K}$, which leads by (\ref{e.range}) to $f(\Omega)=\tilde h(\Omega)\subset h(\Omega\setminus K)=f(\Omega\setminus K)\subset f(\Omega)$. Hence $f(\Omega)=f(\Omega\setminus K)$.
\end{proof}

\begin{remark}
The existence of a holomorphic extension for every map from $\mathcal{H}(\Omega\setminus K)$ depends solely on the connectedness of $\mathbb{C}^n\setminus K$, and hence on the compact set $K$ (the larger open set $\Omega\supset K$ is irrelevant!). Meanwhile, Hartogs companions may be considered regardless of the existence of holomorphic extensions. Therefore, the Hartogs phenomenon for holomorphic maps is a special case of existence of Hartogs companions.
\end{remark}

\subsection{Hartogs companions in dimension one.}\label{ss.dim=1}

The $1$-companions defined below will be used for proving a very general Hartogs-type extension result (Theorem~\ref{t.outer}).

\begin{definition}[Hartogs $1$-companion]\label{d.companion1}
Let an open set $\Omega\subset\mathbb{C}$ and a compact subset $K\subset\Omega$. Any map $f\in\mathcal{H}(\Omega\setminus K)$ has a \emph{Hartogs $1$-companion} $\tilde f\in\mathcal{H}(\Omega)$ defined as follows: for every bounded open set $D\subset\Omega$ with the boundary $\partial D\subset\Omega$ consisting of finitely many piecewise $\mathcal{C}^1$ Jordan curves, $\tilde f|_D$ is given by
\begin{equation}\label{e.companion1}\tilde f(z):=\frac1{2\pi\mathrm{i}}\int_{\partial D}\frac{f(\zeta)}{\zeta-z}\mathrm{d}\zeta,\quad\mbox{for every }z\in D,\end{equation}
where $\partial D$ is oriented such that $D$ lies to the left of $\partial D$.
\end{definition}

The definition is consistent, since for every fixed $a\in\Omega$, the integral defining $\tilde f(a)$ does not depend of the choice of the set $D\supset K\cup\{a\}$. Indeed, let us consider another bounded open set $D_0\subset\Omega$ with the boundary $\partial D_0\subset\Omega$ as in Definition~\ref{d.companion1}, and such that $\overline D\subset D_0$. For the cycles $\partial D_0,\partial D\subset\Omega_a:=\Omega\setminus(K\cup\{a\})$ (we may view these as cycles by identifying any closed path $\gamma:[0,1]\rightarrow\Omega_a$ whose restriction to $[0,1[$ is injective, with the Jordan curve $\gamma([0,1])\subset\Omega_a$), it is easily seen that
\[\mathrm{Ind}_{\partial D_0}(z)-\mathrm{Ind}_{\partial D}(z)=\left\{\begin{array}{ll}1,&z\in D_0\setminus\overline D,\\[0.5mm] 0,&z\in\mathbb{C}\setminus(\overline D_0\setminus D)\supset\mathbb{C}\setminus\Omega_a.\end{array}\right.\]
Applying Cauchy's theorem (the version from Rudin\,\cite{rudin}, Th.\,10.35, p.218) for the map $g\in\mathcal{H}(\Omega_a)$ defined by $g(\zeta)=\frac{f(\zeta)}{\zeta-a}$ yields $\int_{\partial D_0}g(\zeta)\mathrm{d}\zeta=\int_{\partial D}g(\zeta)\mathrm{d}\zeta$. For any bounded open set $D_1\subset\Omega$ with the boundary $\partial D_1\subset\Omega$ as in Definition~\ref{d.companion1}, choosing $D_0\supset\overline D\cup\overline D_1$ now forces $\int_{\partial D_1}g(\zeta)\mathrm{d}\zeta=\int_{\partial D_0}g(\zeta)\mathrm{d}\zeta=\int_{\partial D}g(\zeta)\mathrm{d}\zeta$. This proves the consistency of the definition of the Hartogs 1-companion $\tilde f$.

For arbitrarily fixed $a\in\Omega$ and $D\supset K\cup\{a\}$ as in Definition~\ref{d.companion1}, the set $G:=D-a$ is $(a,1)$-admissible and the restriction $\tilde f|_{\Omega_{G,1}}$ may be represented as in (\ref{e.representation}), with $u=1\in\mathbb{C}$ (such an integral representation is a common feature of the Hartogs companions in any dimension).

Let us note that whenever $f$ has a holomorphic extension to $\Omega$, that is $\tilde f$.

The Hartogs 1-companion of a holomorphic map with finitely many singularities is the regular part of the map (obtained from it by subtracting the sum of the principal parts of all Laurent series about its singularities):

\begin{example}\label{ex.dim=1}
Let an open set $\Omega\subset\mathbb{C}$, a finite subset $S\subset\Omega$, and $f\in\mathcal{H}(\Omega\setminus S)$. For every singularity $s\in S$, let $f_s\in\mathcal{H}(\mathbb{C}\setminus\{s\})$ denote the map defined by the principal part of the Laurent series of $f$ about $s$. Then
\[\tilde f=f-\sum_{s\in S}f_s\quad\mbox{on }\Omega\setminus S.\]
\end{example}

\noindent Indeed, for $g:=f-\sum_{s\in S}f_s\in\mathcal{H}(\Omega\setminus S)$, all singularities from $S$ are removable. Hence for arbitrarily fixed $z\in\Omega\setminus S$ and open set $D\supset S\cup\{z\}$ as in Definition~\ref{d.companion1}, by Cauchy's integral formula it follows that
\[\tilde f(z)=g(z)+\frac1{2\pi\mathrm{i}}\sum_{s\in S}\int_{\partial D}\frac{f_s(\zeta)}{\zeta-z}\mathrm{d}\zeta=g(z),\]
because every map $\zeta\mapsto\frac{f_s(\zeta)}{\zeta-z}$ belongs to $\mathcal{H}(\mathbb{C}\setminus\{z,s\})$ and has residue $0$ at $\infty$.
\vspace{2mm}

For $1$-companions, coincidence sets or range inclusions as in Theorem~\ref{t.C-companion} do not exist. For instance, for $f\in\mathcal{H}(\mathbb{C}\setminus\{0\})$ defined by $f(z)=\frac1z$, according to the above example we have $\tilde f\equiv0$, but $\tilde f(\mathbb{C})\cap f(\mathbb{C}\setminus\{0\})=\emptyset$.

\section{Hartogs companions in arbitrary dimension}\label{s.infinite}

For complex spaces of infinite dimension, various holomorphy types have been considered, however, all imply G\^ateaux holomorphy (see Definition~\ref{d.holomorphy}). Hence any Hartogs-type extension theorem provides at least a G\^ateaux holomorphic extension, which may have some regularity, depending on that of the extended map. Therefore, we first prove extension results for the weakest holomorphy type (G\^ateaux), and then we show the regularity of such extensions. The extension results will follow from the existence of vector-valued Hartogs companions in arbitrary dimension.

\begin{setting}\label{set.infinite}
\emph{From now on, $X$ denotes a complex vector space with $\dim_\mathbb{C}(X)\ge2$, and $Y\ne\{0\}$ a sequentially complete Hausdorff complex locally convex space.}
\end{setting}

In order to preserve generality, we avoid considering additional structures on $X$. Nonetheless, the main results will be illustrated by corollaries stated in a particular, but more familiar setting: for sets in (and maps between) topological vector spaces. We avoid requiring topological compactness (as in Theorems~\ref{t.Hartogs},\ref{t.Kugelsatz},\ref{t.Hartogs=}) in infinite dimension, since in such Hausdorff spaces compact sets have empty interior.

\subsection{Linear cuts, related topologies, and the identity theorem.}\label{ss.cuts}
We next define some appropriate notions; these derive from the vector space structure of $X$ and will allow us to state our results in full generality.

On linear varieties of finite dimension we always consider the (unique) Euclidean topology. For every $d\in\mathbb{N}^*$, let $\Gamma_d(X)$ denote the set of all complex linear varieties $L\subset X$ of dimension $\dim_\mathbb{C}(L)=d$. For arbitrary integers $n\ge d\ge1$, set
\[\Gamma_{d,n}(X):=\bigcup_{k=d}^n\Gamma_k(X),\qquad\Gamma_{d,\infty}(X):=\bigcup_{k\ge d}\Gamma_k(X).\]
A subset $A\subset X$ is said to be \emph{parallel to} $u\in X$ (we write this as $A\parallel u$), if and only if $A+\mathbb{C}\cdot u=A$. If $L\subset X$ is a linear variety, $L\parallel u$ is equivalent to $L+u=L$.

\begin{definition}[linear cuts and related notions]\label{d.d-topology}
Let $d\in\mathbb{N}^*$ and $A\subset X$.
\begin{description}
\item[(a)] For every linear variety $L\subset X$, the intersection $A_L:=A\cap L$ is called the \emph{$L$-cut of $A$}. If $L\in\Gamma_d(X)$, we also call $A_L$ a \emph{$d$-cut of $A$}, and if $a\in A_L$, then $A_L$ is said to be a \emph{$d$-cut of $A$ through $a$}. The set $A$ is called
\begin{description}
\item[(i)] \emph{$d$-open}, if and only if\footnote{To cover the case when $d>\dim_\mathbb{C}(X)$, we need to consider cuts of dimension $d$ and lower.} every cut $A_L$ with $L\in\Gamma_{1,d}(X)$ is open in $L$. In the same way we define \emph{$d$-closed/bounded/compact} sets.
\item[(i')] \emph{finitely open}, if and only if $A$ is $d$-open for every $d\in\mathbb{N}^*$. We define \emph{finitely closed/bounded/compact} sets in the same way.
\item[(ii)] \emph{polygonally connected}, if and only if any to points from $A$ can be joined by a polygonal chain $\Lambda\subset A$.
\end{description}
\item[(b)] A map $f:A\rightarrow Y$ is called \emph{$d$-continuous}, if and only if the restriction $f|_{A_L}$ is continuous, for every $L\in\Gamma_{1,d}(X)$. Set
\[\mathcal{C}_{(d)}(A,Y):=\{f:A\rightarrow Y\,|\,f\mbox{ is $d$-continuous}\},\quad\mathcal{C}_{(d)}(A):=\mathcal{C}_{(d)}(A,\mathbb{C}).\]
\end{description}
\end{definition}

The above notions do not require any topological structure on $X$. Nonetheless, a subset $A\subset X$ is $d$-open/closed, if and only if $A$ is open/closed in the translation invariant topology $\tau_d$ defined by all $d$-open subsets of $X$. Thus $X$ becomes a topological space, which will be denoted by $X_{(d)}$. We have $X_{(d)}=\varinjlim_{L\in\Gamma_{1,d}(X)}L$, where the inductive limit is considered in the category of topological spaces. Any subset $A\subset X$ may be considered as a topological subspace $A_{(d)}$ of $X_{(d)}$. We have
\[\mathcal{C}(A_{(d)},Y)\subset\mathcal{C}_{(d)}(A,Y),\]
with equality if $A$ is $d$-open or $d$-closed.

On $X$ we may also consider the \emph{finite open} topology $\tau_\mathrm{f}$ of the inductive limit $X_\mathrm{f}:=\varinjlim_{L\in\Gamma_{1,\infty}(X)}L$, whose open sets are the finitely open sets. The finite open topology is translation invariant and the multiplication $\mathbb{C}\times X_\mathrm{f}\rightarrow X_\mathrm{f}$ is continuous (see Herv\'e\,\cite{herve}, Prop.\,2.3.4, p.37).

\begin{remark}
In $X_\mathrm{f}$ the origin has a neighborhood base consisting of balanced (hence connected) finitely open sets.

For finitely open sets $\tau_\mathrm{f}$-connectedness is equivalent to polygonal connectedness, and components of finitely open sets are finitely open.

In Hausdorff topological vector spaces, open/closed/bounded/compact sets are finitely open/closed/bounded/compact, but the converse is false (Example~\ref{ex.d-topology}).
\end{remark}

Even in normed spaces, finitely open sets may not be open, and finitely compact sets may not be closed or bounded:

\begin{example}\label{ex.d-topology}
\begin{description}
\item[(a)] For arbitrary infinite set $T$, let us consider the direct sum normed space $\big(\mathbb{C}^{(T)},\|\,\|_\infty\big)$, a map $\rho:T\rightarrow\,]0,\infty[$, and the sets
\begin{eqnarray*}
\Omega_\rho\!\!\!\!&=&\!\!\!\!\big\{u\in\mathbb{C}^{(T)}\,\big|\,|u|<2\rho\mbox{ pointwise}\big\}\subset\mathbb{C}^{(T)},\\
K_\rho\!\!\!\!&=&\!\!\!\!\big\{u\in\mathbb{C}^{(T)}\,\big|\,|u|\le\rho\mbox{ pointwise}\big\}\subset\Omega_\rho.
\end{eqnarray*}
Then $\Omega_\rho$ is finitely open and finitely bounded. Its subset $K_\rho$ is closed, and hence finitely compact. If $\inf\rho(T)=0$, then $\mathring\Omega_\rho=\emptyset$, and so $\Omega_\rho$ is not open. If $\sup\rho(T)=\infty$, then $K_\rho$ is unbounded.
\item[(b)] For arbitrary $p\in\,]0,\infty]$ and $A\subset B_1:=B_\mathbb{C}(0,1)$, in the Fr\'echet space $\ell_\mathbb{C}^p$ (which is a Banach space, if $p\ge1$), the subset\footnote{We take by convention $z^0=1$ for every $z\in\mathbb{C}$.}
\[K_A=\{(z^n)_{n\in\mathbb{N}}\,|\,z\in A\}\subset\ell_\mathbb{C}^p\]
is finitely compact (all $d$-cuts of $K_A$ are finite sets, for every $d\in\mathbb{N}^*$), and $\ell_\mathbb{C}^p\setminus K_A$ is finitely open. Furthermore, $K_A$ is closed (resp. bounded), if and only if $\overline A\cap B_1\subset A$ (resp. $\overline A\subset B_1$). For $A=B_1\cap\mathbb{Q}$, the set $K_A$ is finitely compact, but neither closed, nor bounded, and hence $\ell_\mathbb{C}^p\setminus K_A$ is not open.
\item[(c)] For arbitrary integers $n>d\ge1$, polynomials $p_1,\dots,p_{d+1}\in\mathbb{C}[Z]\setminus\mathbb{C}$ with mutually distinct degrees and $\deg p_1=1$, and set $A\subset\mathbb{C}$, the subset
\[K_A=\{(p_1(z),\dots,p_{d+1}(z),0,\dots,0)\in\mathbb{C}^n\,|\,z\in A\}\subset\mathbb{C}^n\]
is $d$-compact (all $d$-cuts of $K_A$ are finite sets), and $\mathbb{C}^n\setminus K_A$ is $d$-open. Furthermore, $K_A$ is $(d+1)$-closed (resp. $(d+1)$-bounded), if and only if $A$ is closed (resp. bounded). For $A=\mathbb{Q}$, the set $K_\mathbb{Q}$ is $d$-compact, but neither $(d+1)$-closed, nor $(d+1)$-bounded, and hence $\mathbb{C}^n\setminus K_\mathbb{Q}$ is not $(d+1)$-open.
\end{description}
The above (with $d\ge2$ in \textup{(c)}) are examples of sets $K\subset\Omega$ as in Theorem~\ref{t.Y-companion}.
\end{example}

\begin{definition}[G\^ateaux holomorphy, holomorphy, the topologies $\tau_{\mathrm{k}(d)}$ and $\tau_\mathrm{k}$]\label{d.G-holomorphy}{\ }
\begin{description}
\item[(a)] A map $f:\Omega\rightarrow Y$ defined on a $1$-open set\footnote{The usual definition of G\^ateaux holomorphy requires the set $\Omega$ to be finitely open.} $\Omega\subset X$ is called \emph{G\^ateaux holomorphic}, if and only if for all $a\in\Omega$, $v\in X$, and $\varphi\in Y^*$ (the continuous dual of $Y$), there exists $r>0$, such that the map
\[B_\mathbb{C}(0,r)\ni\lambda\mapsto(\varphi\circ f)(a+\lambda v)\in\mathbb{C}\]
is holomorphic. Let $\mathcal{H}_\mathrm{G}(\Omega,Y)$ denote the complex vector space consisting of all such $Y$-valued maps on $\Omega$. Set $\mathcal{H}_\mathrm{G}(\Omega):=\mathcal{H}_\mathrm{G}(\Omega,\mathbb{C})$. Obviously,
\[f\in\mathcal{H}_\mathrm{G}(\Omega,Y)\iff\varphi\circ f\in\mathcal{H}_\mathrm{G}(\Omega) \ \mbox{ for every }\varphi\in Y^*.\]
\item[(b)] If $\Omega$ is an open subset of a Hausdorff complex locally convex space, we may consider the vector spaces of all \emph{holomorphic functions}
\[\mathcal{H}(\Omega,Y):=\{f\in\mathcal{H}_\mathrm{G}(\Omega,Y)\,|\,f\mbox{ is continuous}\},\quad\mathcal{H}(\Omega):=\mathcal{H}(\Omega,\mathbb{C}).\]
\item[(c)] For linear variety $L\subset X$ and $1$-open subset $D\subset L$, we define in the natural way\footnote{Every linear variety is a translated vector subspace.} the vector spaces $\mathcal{H}_\mathrm{G}(D,Y)$ and $\mathcal{H}_\mathrm{G}(D)$. For $L\in\Gamma_{1,\infty}(X)$ and open subset $D\subset L$, we may write these spaces as $\mathcal{H}(D,Y)$ and $\mathcal{H}(D)$, respectively.
\item[(d)] Let a fixed integer $d\ge2$. If $\Omega\subset X$ is $d$-open, then $\mathcal{H}_\mathrm{G}(\Omega,Y)_{\mathrm{k}(d)}$ is a sequentially complete Hausdorff locally convex space, with the topology $\tau_{\mathrm{k}(d)}$ defined by the seminorms
\[p_M:\mathcal{H}_\mathrm{G}(\Omega,Y)\rightarrow\mathbb{R_+},\qquad p_M(f)=\sup_{x\in M}p(f(x))=\max p(f(M)),\]
considered for all continuous seminorms $p:Y\rightarrow\mathbb{R_+}$ and compact subsets $M\subset\Omega_L$, with $L\in\Gamma_{2,d}(X)$. If $Y$ is complete, then so is $\mathcal{H}_\mathrm{G}(\Omega,Y)_{\mathrm{k}(d)}$.
\item[(e)] If $\Omega\subset X$ is finitely open, then $\mathcal{H}_\mathrm{G}(\Omega,Y)_\mathrm{k}$ is a sequentially complete Hausdorff locally convex space, with the topology $\tau_\mathrm{k}$ of uniform convergence on all finite dimensional compact subsets $M\subset\Omega$. This topology is defined by all seminorms as in \textup{(d)}, with $L\in\Gamma_{2,\infty}(X)$. If $Y$ is complete, then so is $\mathcal{H}_\mathrm{G}(\Omega,Y)_\mathrm{k}$.
\end{description}
\end{definition}

If $\Omega\subset X$ is $d$-open, then $\mathcal{H}_\mathrm{G}(\Omega,Y)\subset\mathcal{C}_{(d)}(\Omega,Y)$ (Dineen\,\cite{dineen1}, Lemma~2.3, p.54), and so the seminorms $p_M$ from the above definition are well-defined. Every map $f\in\mathcal{H}_\mathrm{G}(\Omega,Y)$ is G\^ateaux differentiable, that is, the limit $\lim_{\mathbb{C}\ni\lambda\to0}\frac{f(a+\lambda v)-f(a)}\lambda$ exists in $Y$ for all $a\in\Omega$ and $v\in X$ (Dineen\,\cite{dineen2}, Lemma~3.3, p.149). The vector space $\mathcal{H}_\mathrm{G}(\Omega,Y)$ is also a $\mathcal{H}_\mathrm{G}(\Omega)$-module.

The construction of Hartogs companions in arbitrary dimension (Theorem~\ref{t.Y-companion}) and the proofs of the generalized Kugelsatz and of several Hartogs-type extension results require a more special identity theorem.

\begin{theorem}[identity]\label{t.identity}
Let a polygonally connected $2$-open set $\Omega\subset X$ and a subset $C\subset\Omega$, such that $C-c$ is real-absorbing\footnote{For every $x\in X$, there exists $\varepsilon>0$, such that $[0,\varepsilon]\cdot x\subset C-c$ (this holds if $C$ is $1$-open).} for some $c\in C$. Then
\[f(\Omega)-f(c)\subset\overline{\mathrm{Sp}}(f(C)-f(c)),\quad\mbox{for every }f\in\mathcal{H}_\mathrm{G}(\Omega,Y)\]
(in particular, $f\equiv0$ if and only if $f|_C\equiv0$).
\end{theorem}

\begin{proof}
Let us first prove the equivalence. Assume $f|_C\equiv0$. For $c\in C$ as in the theorem, let us consider a linear segment $[c,a]\subset\Omega$. We claim that $f|_A\equiv0$ for some $1$-open subset $A\subset\Omega$, such that $a\in A$ (then $A-a$ is real-absorbing). Set
\[A:=\{x\in X\,|\,[c,x]\subset\Omega\}.\]
Hence $a\in A\subset\Omega$. Suppose there exists a $1$-cut $A_L$, which is not open in $L\in\Gamma_1(X)$. Consequently, there exists a sequence $(x_n)_{n\in\mathbb{N}}\subset L\setminus A$, which converges in $L$ to some $a_0\in A_L\subset A$. Therefore, $[c,a_0]\subset\Omega$. For every $n\in\mathbb{N}$ we have $x_n\notin A$, that is, $\xi_n:=(1-t_n)c+t_nx_n\notin\Omega$ for some $t_n\in[0,1]$. By taking convergent subsequences if necessary, we may assume $\lim_{n\to\infty}t_n=s\in[0,1]$. Clearly, $L\cup\{c\}\subset L'$ for some $L'\in\Gamma_2(X)$. As $\Omega_{L'}$ is open in $L'$ and $(\xi_n)_{n\in\mathbb{N}}\subset L'\setminus\Omega$, a passage to the limit in $L'$ yields $(1-s)c+sa_0\in L'\setminus\Omega$, which contradicts $[c,a_0]\subset\Omega$. We conclude that $A$ is $1$-open, and hence that $A-a$ is a real-absorbing set. In order to show that $f|_A\equiv0$, let us fix $x\in A$ and $\theta\in Y_\mathbb{R}^*$. Thus $[c,x]\subset\Omega$ and the map
\[g:[0,1]\rightarrow\mathbb{R},\qquad g(t)=(\theta\circ f)((1-t)c+tx),\]
is real-analytic. As $C-c$ is a real-absorbing set, we have $[0,\varepsilon]\cdot(x-c)\subset C-c$ for some $\varepsilon\in\,]0,1]$, and so $g|_{[0,\varepsilon]}\equiv0$. By the identity theorem for real-analytic maps we get $g|_{[0,1]}\equiv0$. It follows that $\theta(f(x))=0$ for every $\theta\in Y_\mathbb{R}^*$, which yields $f(x)=0$. Our claim is proved. Since $\Omega$ is polygonally connected, an easy induction (on the number of linear segments from a polygonal chain in $\Omega$ joining $c$ to other points $x\in\Omega$) based on the above claim shows that $f\equiv0$. We thus have proved the equivalence. In order to show the inclusion for $f(\Omega)$, let the map $g:=f-f(c)\in\mathcal{H}_\mathrm{G}(\Omega,Y)$, the closed vector subspace $Y_0:=\overline{\mathrm{Sp}}(g(C))\subset Y$, the quotient Hausdorff locally convex space $Y/Y_0=\{\hat y\,|\,y\in Y\}$, and the standard continuous linear surjection $s:Y\rightarrow Y/Y_0$. Since $s\circ g\in\mathcal{H}_\mathrm{G}(\Omega,Y/Y_0)$ and $(s\circ g)|_C\equiv\hat0$, by the already proved equivalence we get $s\circ g\equiv\hat0$, that is, $g(\Omega)\subset Y_0$. This yields $f(\Omega)-f(c)\subset\overline{\mathrm{Sp}}(f(C)-f(c))$.
\end{proof}

The identity theorem now allows us to prove a much more general vector-valued version of the maximum modulus principle.

\begin{theorem}[maximum norm principle]\label{t.maximum}
Assume $Y$ is a strictly convex normed space. Let a polygonally connected $2$-open set $\Omega\subset X$ and a map $f\in\mathcal{H}_\mathrm{G}(\Omega,Y)$. If\footnote{Here $\|f(\cdot)\|$ denotes the map $\Omega\ni x\mapsto\|f(x)\|\in\mathbb{R}_+$.} $\|f(\cdot)\|$ has a $\tau_{(1)}$-local maximum, then $f$ is constant.
\end{theorem}

\begin{proof}
We have divided the proof into two steps.\\
\emph{Step 1}. Let us first show the theorem for $Y=\mathbb{C}$. For $f\in\mathcal{H}_\mathrm{G}(\Omega)$, assume that $|f|$ has a $\tau_{(1)}$-local maximum at $c\in\Omega$. Set $\Gamma_c=\{L\in\Gamma_1(X)\,|\,c\in L\}$. For every $L\in\Gamma_c$, we have $|f(c)|=\max_{x\in C_L}|f(x)|$ for some connected open neighborhood $C_L\subset\Omega_L$ of $c$ in $L$. As $f|_{C_L}\in\mathcal{H}(C_L)$, by the classical maximum modulus principle we deduce that $f|_{C_L}\equiv f(c)$. Set $C:=\bigcup_{L\in\Gamma_c}C_L\subset\Omega$. Since $C-c$ is an absorbing set and $f|_C\equiv f(c)$, Theorem~\ref{t.identity} now yields $f\equiv f(c)$.\\
\emph{Step 2}. In the general case, assume $\|f(\cdot)\|$ has a $\tau_{(1)}$-local maximum at $c\in\Omega$ and set $y_0:=f(c)$. There is a $1$-open neighborhood $C\subset\Omega$ of $c$ in $X_{(1)}$, such that
\begin{equation}\label{e.max}\|f(x)\|\le\|y_0\|,\quad\mbox{for every }x\in C.\end{equation}
Hence $C-c$ is an absorbing set. According to the Hahn-Banach theorem, there exists $\varphi\in Y^*\setminus\{0\}$, such that $\varphi(y_0)=\|y_0\|$ and $|\varphi(y)|\le\|y\|$ for every $y\in Y$. Set $g:=\varphi\circ f\in\mathcal{H}_\mathrm{G}(\Omega)$. For every $x\in C$, we have
\[|g(x)|=|\varphi(f(x))|\le\|f(x)\|\le\|y_0\|=|\varphi(y_0)|=|g(c)|,\]
and so $|g|$ has a $\tau_{(1)}$-local maximum at $c$. By the already proved theorem for $Y=\mathbb{C}$ it follows that $g\equiv g(c)$, that is, $f(\Omega)\subset y_0+\ker\varphi$. This together with (\ref{e.max}) and the strict convexity of $Y$ lead to $f(C)\subset(y_0+\ker\varphi)\cap\overline B_Y(0,\|y_0\|)=\{y_0\}$, and hence to $f|_C\equiv y_0$. By Theorem~\ref{t.identity} we conclude that $f\equiv y_0$.
\end{proof}

In the above theorem the strict convexity of $Y$ cannot be dropped:

\begin{example}\label{ex.maximum}
Let us consider a complex vector space $Y$ of dimension at least $2$, a norm $\nu:Y\rightarrow\mathbb{R}_+$, a vector $y_0\in Y$ and $\varphi\in Y^*$ such that $\varphi(y_0)=1$, and the equivalent norm on $Y$ defined by $\|y\|=\max\{\nu(y-\varphi(y)y_0),|\varphi(y)|\}$. For fixed $n\in\mathbb{N}^*$, let $T\in\mathcal{L}(\mathbb{C}^n,\ker\varphi)\setminus\{0\}$ and $f\in\mathcal{H}(\mathbb{C}^n,Y)$ defined by $f(x)=T(x)+y_0$. Then $\|f(\cdot)\|=\max\{\nu\circ T,1\}$ has a local maximum and a global minimum at $c=0$. Nonetheless, both maps $f$ and $\|f(\cdot)\|$ are nonconstant.
\end{example}

Without the strict convexity of $Y$ from Theorem~\ref{t.maximum}, we still can show that if $\|f(\cdot)\|$ has a $\tau_{(1)}$-local maximum $c\in\Omega$, then $c$ is also a global minimum and $f(\Omega)$ has empty interior (and therefore is not a domain), and if $f(c)\ne0$, then $f$ vanishes nowhere on $\Omega$. This still holds with the norm replaced by a continuous seminorm:

\begin{theorem}[max-min seminorm principle]\label{t.max-min}
Let a polygonally connected $2$-open set $\Omega\subset X$, a map $f\in\mathcal{H}_\mathrm{G}(\Omega,Y)$, and a continuous seminorm $p\not\equiv0$ on $Y$. Assume that $c\in\Omega$ is a $\tau_{(1)}$-local maximum for $p\circ f$. Then $c$ is also a global minimum and $\mathring{\overline{f(\Omega)}}=\emptyset$. In particular, if $p(f(c))>0$, then $0\notin f(\Omega)$.
\end{theorem}

\begin{proof}
Let us consider the quotient normed space $\hat Y:=Y/{p^{-1}(\{0\})}=\{\hat y\,|\,y\in Y\}$ with the usual norm $\|\hat y\|=p(y)$, the standard linear surjection $s:Y\rightarrow\hat Y$, the map $\hat f:=s\circ f\in\mathcal{H}_\mathrm{G}(\Omega,\hat Y)$, and $y_0:=f(c)$. Thus $\|\hat f(\cdot)\|$ has a $\tau_{(1)}$-local maximum at $c$. As in the proof of Theorem~\ref{t.maximum} (Step 2) we choose $\varphi\in\hat Y^*\setminus\{0\}$, for which we show that $\hat f(\Omega)\subset\hat y_0+\ker\varphi$. Consequently, for every $x\in\Omega$ we have
\[p(f(x))=\|\hat f(x)\|\ge|\varphi(\hat f(x))|=|\varphi(\hat y_0)|=\|\hat y_0\|=p(y_0).\]
Hence $p\circ f$ has a global minimum at $c$. The above inclusion for $\hat f(\Omega)$ is equivalent to $f(\Omega)\subset y_0+\ker(\varphi\circ s)$. As $\varphi\circ s\in Y^*\setminus\{0\}$, we have $\mathring{\overline{f(\Omega)}}=y_0+\mathring{\overline{\ker(\varphi\circ s)}}=\emptyset$. If $p(f(c))>0$, then $0\notin p(f(\Omega))$, and so $0\notin f(\Omega)$.
\end{proof}

\begin{corollary}[minimum modulus principle]\label{c.minimum}
Let a polygonally connected $2$-open set $\Omega\subset X$ and a map $f\in\mathcal{H}_\mathrm{G}(\Omega)$. If $|f|$ has a $\tau_{(1)}$-local minimum at $c\in\Omega$ and $f(c)\ne0$, then $f$ is constant.
\end{corollary}

\begin{proof}
Let us consider $\Gamma_c$ as in the proof of Theorem~\ref{t.maximum}. For every $L\in\Gamma_c$, we have $|f(c)|=\min_{x\in C_L}|f(x)|$ for some connected open neighborhood $C_L\subset\Omega_L$ of $c$ in $L$. As $g:=\frac1{f|_{C_L}}\in\mathcal{H}(C_L)$ and $|g|$ has a maximum at $c$, by the classical maximum modulus principle we get $g\equiv g(c)$, that is, $f|_{C_L}\equiv f(c)$. As in the proof of Theorem~\ref{t.maximum} (Step 1) it follows that $f\equiv f(c)$.
\end{proof}

\subsection{Range inclusions and inertia.}\label{ss.range}

The next result on automatic extensions will be used for proving the range inclusions for Hartogs companions and the continuity of the Hartogs linear operator, as well as for obtaining boundary principles.

\begin{proposition}[range inclusions for G\^ateaux holomorphic extensions]\label{p.range}
Let two $1$-open sets $\Omega_0,\Omega\subset X$ and $C\subset\Omega_0\cap\Omega$. Assume that\footnote{This condition is similar to that from Definition~\ref{d.domain}(b).} for every $f\in\mathcal{H}_\mathrm{G}(\Omega_0,Y)$, the restriction $f|_C$ has a unique extension $\bar f\in\mathcal{H}_\mathrm{G}(\Omega,Y)$. Then
\begin{description}
\item[(a)] The same extension property holds for $\mathbb{C}$-valued maps and
\begin{eqnarray*}
\overline{(h\cdot\!f)}=\bar h\cdot\!\bar f&&\!\!\!\!\!\mbox{for all }h\in\mathcal{H}_\mathrm{G}(\Omega_0),\,f\in\mathcal{H}_\mathrm{G}(\Omega_0,Y),\\
\bar h(\Omega)\subset h(\Omega_0)&&\!\!\!\!\!\mbox{for every }h\in\mathcal{H}_\mathrm{G}(\Omega_0),\\
\bar f(\Omega)\subset\overline{\mathrm{co}}(f(\Omega_0))&&\!\!\!\!\!\mbox{for every }f\in\mathcal{H}_\mathrm{G}(\Omega_0,Y).
\end{eqnarray*}
\item[(b)] If $Y$ is a unital Banach algebra (not necessarily commutative), then
\[H:\mathcal{H}_\mathrm{G}(\Omega_0,Y)\rightarrow\mathcal{H}_\mathrm{G}(\Omega,Y),\qquad H(f)=\bar f,\]
is a unital $\mathbb{C}$-algebra morphism. For every $f\in\mathcal{H}_\mathrm{G}(\Omega_0,Y)$, we have
\[f(\Omega_0)\subset\mathrm{U}(Y) \ \Longrightarrow \ \bar f(\Omega)\subset\mathrm{U}(Y),\]
where $\mathrm{U}(Y)$ denotes the (open) set of all invertible elements of $Y$.
\end{description}
\end{proposition}

\begin{proof}
(a). The proof is divided into three steps.\\
\emph{Step 1}. We first prove the existence and uniqueness property for $\mathbb{C}$-valued maps. Let us fix $h\in\mathcal{H}_\mathrm{G}(\Omega_0)$. Choose $y\in Y\setminus\{0\}$ and $\varphi\in Y^*$, such that $\varphi(y)=1$. For $f:=h\cdot y\in\mathcal{H}_\mathrm{G}(\Omega_0,Y)$ and the unique corresponding map $\bar f\in\mathcal{H}_\mathrm{G}(\Omega,Y)$ as in the hypothesis, set $\bar h:=\varphi\circ\bar f\in\mathcal{H}_\mathrm{G}(\Omega)$. Since $\varphi\circ f=h$, it follows that
\[\bar h|_C=\varphi\circ\bar f|_C=\varphi\circ f|_C=(\varphi\circ f)|_C=h|_C.\]
The existence of $\bar h\in\mathcal{H}_\mathrm{G}(\Omega)$ is proved. For the uniqueness part, let $g\in\mathcal{H}_\mathrm{G}(\Omega)$, such that $g|_C=h|_C$. For $g_y:=g \cdot y\in\mathcal{H}_\mathrm{G}(\Omega,Y)$, we see that
\[g_y|_C=g|_C\cdot y=h|_C\cdot y=f|_C=\bar f|_C.\]
By the uniqueness of $\bar f$ we get $g_y=\bar f$, which yields $g=\varphi\circ g_y=\varphi\circ\bar f=\bar h$. Hence $\bar h$ is unique. We thus have proved the claimed existence and uniqueness property. For arbitrary $h\in\mathcal{H}_\mathrm{G}(\Omega_0)$ and $f\in\mathcal{H}_\mathrm{G}(\Omega_0,Y)$, and the unique corresponding maps $\bar h\in\mathcal{H}_\mathrm{G}(\Omega)$ and $\bar f\in\mathcal{H}_\mathrm{G}(\Omega,Y)$, we have $\overline{(h\cdot\!f)}|_C=(h\cdot f)|_C=(\bar h\cdot\!\bar f)|_C$, which leads by the uniqueness property to $\overline{(h\cdot\!f)}=\bar h\cdot\!\bar f$.\\
\emph{Step 2}. We next show the inclusion for $\bar h(\Omega)$ ($h\in\mathcal{H}_\mathrm{G}(\Omega_0)$). Let us fix $\lambda\in\mathbb{C}\setminus h(\Omega_0)$. Thus $g:=\frac1{h-\lambda}\in\mathcal{H}_\mathrm{G}(\Omega_0)$ and $f:=(h-\lambda)g\equiv1$. For the unique corresponding maps $\bar h,\bar g,\bar f\in\mathcal{H}_\mathrm{G}(\Omega)$ and for $F:=(\bar h-\lambda)\bar g\in\mathcal{H}_\mathrm{G}(\Omega)$, we have $F|_C=f|_C=\bar f|_C$. Since $f\equiv1$, the uniqueness of $\bar f$ yields $F=\bar f\equiv1$. Hence $\lambda\in\mathbb{C}\setminus\bar h(\Omega)$. As $\lambda$ was arbitrary, we conclude that $\bar h(\Omega)\subset h(\Omega_0)$.\\
\emph{Step 3}. We next show the inclusion for $\bar f(\Omega)$ ($f\in\mathcal{H}_\mathrm{G}(\Omega_0,Y)$). According to Step 1, for every $\varphi\in Y^*$, the unique map $\bar h\in\mathcal{H}_\mathrm{G}(\Omega)$ with the property that $\bar h|_C=(\varphi\circ f)|_C$ is $\bar h:=\varphi\circ\bar f$. By the inclusion from Step 2 we see that $\varphi(\bar f(\Omega))=\bar h(\Omega)\subset\varphi(f(\Omega_0))$. Hence $\theta(\bar f(\Omega))\subset\theta(f(\Omega_0))$ for every $\theta\in Y_\mathbb{R}^*$ (the continuous dual of $Y$ considered as a real locally convex space $Y_\mathbb{R}$), because $Y_\mathbb{R}^*=\{\mathrm{Re}(\varphi)\,|\,\varphi\in Y^*\}$. Now by the Hahn-Banach separation theorem it follows that $\bar f(a)\in\overline{\mathrm{co}}(f(\Omega_0))$ for every $a\in\Omega$, that is, $\bar f(\Omega)\subset\overline{\mathrm{co}}(f(\Omega_0))$.\\
(b). Let $e$ denote the identity element of $Y$. Since all G\^ateaux holomorphic maps are also G\^ateaux differentiable, both $\mathcal{H}_\mathrm{G}(\Omega_0,Y)$ and $\mathcal{H}_\mathrm{G}(\Omega_0,Y)$ are unital algebras. That $H$ is a unital algebra morphism is immediate, by the uniqueness property. Let $f\in\mathcal{H}_\mathrm{G}(\Omega_0,Y)$, such that $f(\Omega_0)\subset\mathrm{U}(Y)$. For $g\in\mathcal{H}_\mathrm{G}(\Omega_0,Y)$ defined by $g(x)=f(x)^{-1}$, we have $f\cdot g=g\cdot f\equiv e$. By applying the morphism $H$ we deduce that $\bar f\cdot\bar g=\bar g\cdot\bar f\equiv e$, which yields $\bar f(\Omega)\subset\mathrm{U}(Y)$.\\
\emph{Alternative proof of the range inclusion \textup{(\ref{e.range})} from Theorem~\ref{t.C-companion}(a)}. We have already proved that for every $f\in\mathcal{H}_\mathrm{G}(\Omega\setminus K)$, there is a unique $\tilde f\in\mathcal{H}_\mathrm{G}(\Omega)$ (the Hartogs companion of $f$), such that $\tilde f|_{C_{K,\Omega}}=f|_{C_{K,\Omega}}$. Since $C:=C_{K,\Omega}\subset \Omega_0:=\Omega\setminus K\subset\Omega$, by the first inclusion from Proposition~\ref{p.range}(a) we get $\tilde f(\Omega)\subset f(\Omega\setminus K)$.
\end{proof}

For vector-valued Hartogs companions the range inclusion corresponding to (\ref{e.range}) will have three versions, depending on the dimension of $Y$. In order to state it without mentioning cases, we next make an appropriate convention.

\begin{notation}\label{n.range}
For arbitrary subsets $A,B\subset Y$, we write $A\sqsubset B$, if and only if one of the following three conditions holds:
\begin{description}
\item[(a)] $A\subset\overline{\mathrm{co}}(B)$ and $Y$ has infinite dimension.
\item[(b)] $A\subset\mathrm{co}(B)$ and $Y$ has finite dimension at least $2$.
\item[(c)] $A\subset B$ and $Y$ has dimension $1$.
\end{description}
The preorder ``$\sqsubset$'' on $\mathcal{P}(Y)$ is weaker than the inclusion. If $A_i\sqsubset B_i$ for every $i\in I$, then $\bigcup_{i\in I}A_i\sqsubset\bigcup_{i\in I}B_i$.
\end{notation}

The following vector-valued version of Theorem~\ref{t.C-companion} is a needed ingredient for the construction of Hartogs companions in arbitrary dimension.

\begin{lemma}[vector-valued Hartogs companion]\label{l.Y-companion}
Let an open set $\Omega\subset\mathbb{C}^n$ and a compact subset $K\subset\Omega$. Theorem~\ref{t.C-companion} holds for vector-valued maps $f\in\mathcal{H}(\Omega\setminus K,Y)$, with all spaces of the form $\mathcal{H}(\cdot)$ replaced by the corresponding $\mathcal{H}(\cdot,Y)$ and with the range inclusion \textup{(\ref{e.range})} for $\tilde f\in\mathcal{H}(\Omega,Y)$ replaced by
\begin{equation}\label{e.Y-range}\tilde f(\Omega)\sqsubset f(\Omega\setminus K).\end{equation}
\end{lemma}

\begin{proof}
The statements of the lemma corresponding to those of Theorem~\ref{t.C-companion} will be referred to as Lemma~\ref{l.Y-companion}(a,b,c,d).\\
(a). According to Theorem~\ref{t.C-companion}(a), for every $\varphi\in Y^*$ the map $f_\varphi:=\varphi\circ f\in\mathcal{H}(\Omega\setminus K)$ has a unique Hartogs companion $\tilde f_\varphi\in\mathcal{H}(\Omega)$. Let $u\in\mathbb{C}^n\setminus\{0\}$ be fixed. For every $a\in\Omega$, choose a particular $(a,u)$-admissible set $G_a\subset\mathbb{C}$ and define
\[\tilde f(a)=\frac1{2\pi\mathrm{i}}\int_{\partial G_a}\frac{f(a+\zeta u)}\zeta\mathrm{d}\zeta\in Y\]
(the integral exists, since $f$ is continuous, $\partial G_a$ is piecewise $\mathcal{C}^1$, and $Y$ is sequentially complete). We thus have defined a map $\tilde f:\Omega\rightarrow Y$. By Theorem~\ref{t.C-companion}(c) we see that
\[\tilde f_\varphi(a)=\frac1{2\pi\mathrm{i}}\int_{\partial G_a}\frac{f_\varphi(a+\zeta u)}\zeta\mathrm{d}\zeta=\varphi\big(\tilde f(a)\big),\quad\mbox{for all }\varphi\in Y^*,\,a\in\Omega.\]
Hence $\varphi\circ\tilde f=\tilde f_\varphi\in\mathcal{H}(\Omega)$ for every $\varphi\in Y^*$, and so $\tilde f\in\mathcal{H}(\Omega,Y)$. Since every $\varphi\circ\tilde f$ is the Hartogs companion of $\varphi\circ f$, it follows that (\ref{e.companion}) holds. The uniqueness of $\tilde f\in\mathcal{H}(\Omega,Y)$ satisfying (\ref{e.companion}) is obvious. The inclusion (\ref{e.Y-range}) will be proved after (d).\\
(b). For every $\varphi\in Y^*$, by (a) and Theorem~\ref{t.C-companion}(b) we see that $\varphi\circ\tilde f|_{\Omega_0}$ is the Hartogs companion of $\varphi\circ f|_{\Omega_0\setminus K_0}$. Therefore, $\tilde f|_{\Omega_0}$ is the Hartogs companion of $f|_{\Omega_0\setminus K_0}$.\\
(c). By Theorem~\ref{t.C-companion}(c), (\ref{e.representation}) holds for all $\varphi\circ\tilde f=\tilde f_\varphi$. Therefore, (\ref{e.representation}) also holds for $\tilde f$.\\
(d). By Theorem~\ref{t.C-companion}(d) we get $C_{K,\Omega}=\Omega\setminus K$, which yields $\tilde f|_{\Omega\setminus K}=f$, by (\ref{e.companion}).\\
\emph{Proof of} (\ref{e.Y-range}). Applying Proposition~\ref{p.range}(a) for $C:=C_{K,\Omega}\subset\Omega_0:=\Omega\setminus K\subset\Omega$ yields $\tilde f(\Omega)\subset\overline{\mathrm{co}}(f(\Omega\setminus K))$, for arbitrary dimension of $Y$. If $\dim_\mathbb{C}(Y)=1$, then (\ref{e.range}) holds. For $\dim_\mathbb{C}(Y)\in\mathbb{N}^*\setminus\{1\}$, in order to show that $\tilde f(\Omega)\subset\mathrm{co}(f(\Omega\setminus K))$, let us fix $a\in\Omega$. Choose the sets $K_0\subset\Omega_0$ as in Theorem~\ref{t.C-companion}(b), such that $a\in\Omega_0$ and $K\subset\mathring K_0$, and $\overline\Omega_0\subset\Omega$ is compact. Since $\tilde f|_{\Omega_0}$ is the Hartogs companion of $f|_{\Omega_0\setminus K_0}$ and $\overline\Omega_0\setminus\mathring K_0$ is compact, by the already proved part of (\ref{e.Y-range}) we get
\[\tilde f(a)\in\tilde f(\Omega_0)\subset\overline{\mathrm{co}}(f(\Omega_0\setminus K_0))\subset\overline{\mathrm{co}}\big(f\big(\overline\Omega_0\setminus\mathring K_0\big)\big)=\mathrm{co}\big(f\big(\overline\Omega_0\setminus\mathring K_0\big)\big)\subset\mathrm{co}(f(\Omega\setminus K)).\]
As $a$ was arbitrary, we conclude that $\tilde f(\Omega)\subset\mathrm{co}(f(\Omega\setminus K))$.
\end{proof}

We consider domains of holomorphy as in Herv\'e\,\cite{herve} (Def.\,5.2.1(a), p.135), but without the connectedness assumption on the set:

\begin{definition}[domain of holomorphy]\label{d.domain}
An open subset $D\subset Y$ is called a \emph{domain of holomorphy}, if and only if there are no open sets $C,D_1\subset Y$, such that
\begin{description}
\item[(a)] $\emptyset\ne C\subset D\cap D_1$ and $D_1$ is connected, with $D_1\not\subset D$.
\item[(b)] For every $g\in\mathcal{H}(D)$, the restriction $g|_C$ has an extension $\bar g\in\mathcal{H}(D_1)$.
\end{description}
\end{definition}

\begin{theorem}[range inertia]\label{t.inertia}
Let an open set $\Omega\subset\mathbb{C}^n$, a compact subset $K\subset\Omega$, and $f\in\mathcal{H}(\Omega\setminus K,Y)$. Then for every domain of holomorphy $D\subset Y$, we have\footnote{If $f$ has an extension from $\mathcal{H}(\Omega,Y)$ (which must be $\tilde f$), then equivalence holds in (\ref{e,inertia}).}
\begin{equation}\label{e,inertia}f(\Omega\setminus K)\subset D \ \Longrightarrow \ \tilde f(\Omega)\subset D.\end{equation}
\end{theorem}

\begin{proof}
Let a domain of holomorphy $D\subset Y$, such that $f(\Omega\setminus K)\subset D$. Assume $\tilde f(b)\in Y\setminus D$ for some $b\in\Omega$. Let us choose sets $K_0$ and $\Omega_0$ as in Lemma~\ref{l.Y-companion}(b), such that $b\in\Omega_0$ and $K\subset\mathring K_0$, and $\overline\Omega_0\subset\Omega$ is compact. The Hartogs companion of the restriction $f_0:=f|_{\Omega_0\setminus K_0}$ is $\tilde f_0=\tilde f|_{\Omega_0}$, by Lemma~\ref{l.Y-companion}(b). Since the set $f\big(\overline\Omega_0\setminus\mathring K_0\big)\subset f(\Omega\setminus K)\subset D$ is compact and $D$ is open, there is an absolutely convex open neighborhood $V_1\subset Y$ of $0$, such that $f_0(\Omega_0\setminus K_0)+V_1\subset D$. Thus
\[f_0\in\mathcal{H}(\Omega_0\setminus K_0,Y),\qquad f_0\in\mathcal{H}(\Omega_0\setminus K_0)+V_1\subset D,\qquad\tilde f_0(b)\notin D.\]
By replacing $\Omega,K,f,\tilde f$, by $\Omega_0,K_0,f_0,\tilde f_0$, respectively, we can assume that
\[f(\Omega\setminus K)+V_1\subset D.\]
There exist a component $\omega\in\Upsilon_\Omega$ containing $b$ and a path $\gamma:[0,1]\rightarrow\omega$, such that $c:=\gamma(0)\in\omega\cap K_\omega^\mathrm{u}\subset C_{K,\Omega}$ and $b=\gamma(1)$. Since $\tilde f(c)=f(c)\in D$ and $\tilde f(b)\in Y\setminus D$, there exists $s\in\,]0,1[$, such that $\tilde f(\gamma([0,s]))\subset D$ and $\tilde f(\gamma(s))+V_1\not\subset D$. Set
\[\Gamma:=\gamma([0,s])\subset\omega,\quad a:=\gamma(s)\in\Gamma,\quad y_0:=\tilde f(a)\in D,\quad D_1:=y_0+V_1\not\subset D.\]
Thus $D_1$ is a connected open set and $\tilde f(\Gamma)\subset D$. As $\Gamma$ is compact and $D,\omega$ are open, we can choose successively an absolutely convex open neighborhood $V_0\subset V_1$ of $0$ and an open ball $B\subset\mathbb{C}^n$ centered at the origin, such that
\[\tilde f(\Gamma)+2V_0\subset D,\qquad\omega_0:=\Gamma+B\subset\omega\cap\tilde f^{-1}(\tilde f(\Gamma)+V_0).\]
Hence $\omega_0\subset\omega$ is a connected open set with $\tilde f(\omega_0)+V_0\subset D$ and
\[C:=y_0+V_0\subset D\cap D_1\]
is a nonempty open set. We next show that the condition from Definition~\ref{d.domain}(b) holds (which will lead to a contradiction). To this end, let us fix $g\in\mathcal{H}(D)$. For every fixed $z\in V_1$, we have $f_z:=f+z\in\mathcal{H}(\Omega\setminus K,Y)$ and $f_z(\Omega\setminus K)\subset D$. As $g\in\mathcal{H}(D)$, we also have $h_z:=g\circ f_z\in\mathcal{H}(\Omega\setminus K)$. According to Theorem~\ref{t.C-companion}(a), $h_z$ has a unique Hartogs companion $\tilde h_z\in\mathcal{H}(\Omega)$. Hence $\tilde h_z|_{C_{K,\Omega}}=h_z|_{C_{K,\Omega}}=g\circ\tilde f_z|_{C_{K,\Omega}}$. For every $z\in V_0$, by the identity theorem we get $\tilde h_z|_{\omega_0}=g\circ\tilde f_z|_{\omega_0}$ (the latter is well-defined and the two maps coincide on the open set $\omega_0\cap C_{K,\Omega}\ni c$). Thus
\begin{equation}\label{e.h_z}\tilde h_z(x)=g\big(\tilde f(x)+z\big),\quad\mbox{for all }z\in V_0,\,x\in\omega_0.\end{equation}
Now let us define the map
\[\bar g:D_1\rightarrow\mathbb{C},\qquad\bar g(y)=\tilde h_{y-y_0}(a).\]
For every $y\in C$, since $z_y:=y-y_0\in V_0$, by (\ref{e.h_z}) and the definition of $\bar g$ we see that $\bar g(y)=\tilde h_{z_y}(a)=g(y_0+z_y)=g(y)$. Hence $\bar g|_C=g|_C$. In order to show that $\bar g\in\mathcal{H}(D_1)$, let us fix $u\in\mathbb{C}^n\setminus\{0\}$ and an $(a,u)$-admissible set $G\subset\mathbb{C}$. By using the representation formula (\ref{e.representation}) (at $x=a$) we deduce that
\begin{equation}\label{e.bar g}\bar g(y)=\frac1{2\pi\mathrm{i}}\int_{\partial G}\frac{h_{y-y_0}(a+\zeta u)}\zeta\mathrm{d}\zeta=\frac1{2\pi\mathrm{i}}\int_{\partial G}\frac{g(f(a+\zeta u)+y-y_0)}\zeta\mathrm{d}\zeta,\end{equation}
for every $y\in D_1$. Since the map
\[h:(\Omega\setminus K)\times V_1\rightarrow\mathbb{C},\qquad h(x,y)=g(f(x)+y-y_0),\]
is holomorphic, it follows that $\bar g$ is continuous and G\^ateaux $y$-differentiation under the integral sign holds in (\ref{e.bar g}). Hence $\bar g\in\mathcal{H}(D_1)$. As the existence of the sets $D_1$ and $C$ as above leads to a contradiction, we conclude that $\tilde f(\Omega)\subset D$.
\end{proof}

\subsection{Boundary principle and 2-cuts properties}\label{ss.boundary}

In this section we explore the first consequences of Lemma~\ref{l.Y-companion} and Theorem~\ref{t.inertia}. We will be mainly using $2$-cuts, since topological assumptions on these are less restrictive and the $2$-boundary of any set (defined below) is smaller.

\begin{proposition}[$2$-compact excision]\label{p.excision}
Let a $2$-open set $\Omega\subset X$ and $f\in\mathcal{H}_\mathrm{G}(\Omega,Y)$. Let a subset $K\subset\Omega$ which has particular compact $2$-cuts through all points\footnote{This condition on $K$ is fulfilled in particular by $2$-compact sets.} of $K$. Then $f(\Omega)\sqsubset f(\Omega\setminus K)$. For every domain of holomorphy $D\subset Y$,
\[f(\Omega)\subset D\iff f(\Omega\setminus K)\subset D.\]
If $Y=\mathbb{C}$, then $f(\Omega)=f(\Omega\setminus K)$.
\end{proposition}

\begin{proof}
Let us fix $a\in K$, together with a compact $2$-cut $K_L$ through $a$. Since $f|_{\Omega_L}\in\mathcal{H}(\Omega_L,Y)$ is the Hartogs companion of its restriction $f|_{\Omega_L\setminus K_L}$, by (\ref{e.Y-range}) it follows that $f(a)\in f(\Omega_L)\sqsubset f(\Omega_L\setminus K_L)\subset f(\Omega\setminus K)$. As $a$ was arbitrary, we conclude that $f(K)\sqsubset f(\Omega\setminus K)$, and hence that $f(\Omega)\sqsubset f(\Omega\setminus K)$. For the equivalence, we only need to prove the implication ``$\Leftarrow$''. Therefore, assuming $f(\Omega\setminus K)\subset D$, we next show that $f(K)\subset D$. Let us fix again $a\in K$, together with a compact $2$-cut $K_L$ through $a$. Since $f|_{\Omega_L}$ is the Hartogs companion of $f|_{\Omega_L\setminus K}$, by Theorem~\ref{t.inertia} we see that $f(a)\in f(\Omega_L)\subset D$. We thus conclude that $f(K)\subset D$.
\end{proof}

\begin{corollary}[excision]\label{c.excision}
Assume $X$ is a Hausdorff topological vector space. Let an open set $\Omega\subset X$ and a $2$-bounded closed subset $K\subset\Omega$. Then
\[f(\Omega)=f(\Omega\setminus K),\quad\mbox{for every }f\in\mathcal{H}_\mathrm{G}(\Omega).\]
\end{corollary}

\begin{proof}
Since $\Omega$ is $2$-open and $K$ is $2$-compact, we may apply Proposition~\ref{p.excision}.
\end{proof}

\begin{corollary}[level sets]\label{c.level}
Let a $2$-open set $\Omega\subset X$, a map $f\in\mathcal{H}_\mathrm{G}(\Omega)$, and $S\subset\mathbb{C}$. Then for every $L\in\Gamma_2(X)$, the set $f^{-1}(S)\cap L$ is not nonempty and relatively compact in $\Omega_L$. In particular, $f^{-1}(\{0\})$ has no nonempty compact $2$-cuts.
\end{corollary}

\begin{proof}
Assume $\emptyset\ne f^{-1}(S)\cap L\subset K\subset\Omega_L$, for some $L\in\Gamma_2(X)$ and compact set $K\subset L$. Since $f|_{\Omega_L}\in\mathcal{H}(\Omega_L)$ extends its restriction $f|_{\Omega_L\setminus K}$, by Corollary~\ref{c.excision} we deduce that $f(f^{-1}(S)\cap L)\subset f(\Omega_L)=f(\Omega_L\setminus K)\subset Y\setminus S$. It follows that  $f(f^{-1}(S)\cap L)\subset S\cap(Y\setminus S)=\emptyset$, which yields $f^{-1}(S)\cap L=\emptyset$.
\end{proof}

For the next theorem, we define the \emph{$2$-boundary} of any subset $A\subset X$ by
\[\partial_2A:=\bigcup_{L\in\Gamma_2(X)}\partial A_L,\]
where each boundary $\partial A_L:=\partial(A\cap L)$ is considered in $L$. If $X$ is also a topological vector space, then
\[\partial_2A\subset\partial A,\qquad A\cup\partial_2A\subset\overline A.\]

\begin{theorem}[$2$-boundary principle]\label{t.boundary}
Let a $1$-open set $\Omega\subset X$ having particular bounded open $2$-cuts through all points\footnote{These conditions on $\Omega$ hold in particular for $2$-open sets which are $2$-bounded.} of $\Omega$ and $f\in\mathcal{H}_\mathrm{G}(\Omega,Y)\cap\mathcal{C}_{(2)}(\Omega\cup\partial_2\Omega,Y)$. Then $f(\Omega)\sqsubset f(\partial_2\Omega)$. For every continuous seminorm $p:Y\rightarrow\mathbb{R_+}$,
\[\sup_{x\in\Omega}p(f(x))=\sup_{x\in\partial_2\Omega}p(f(x)).\]
For every domain of holomorphy $D\subset Y$,
\[f(\partial_2\Omega)\subset D \ \Longrightarrow \ f(\Omega)\subset D.\]
\end{theorem}

\begin{proof}
We first show the inclusion $f(\Omega)\sqsubset f(\partial_2\Omega)$. To this end, let us fix $y=f(a)$, with $a\in\Omega$, together with a bounded $2$-cut $\Omega_L$ through $a$. Let us choose a sequence $(K_n)_{n\in\mathbb{N}}$ of compact subsets of $\Omega_L$, such that $K_n\subset\mathring K_{n+1}$ (the interior is taken in $L$) for every $n\in\mathbb{N}$ and $\bigcup_{n\in\mathbb{N}}K_n=\Omega_L$. According to Proposition~\ref{p.excision}, for every $n\in\mathbb{N}$ we have $y\in f(\Omega_L)\sqsubset f(\Omega_L\setminus K_n)$. We need to analyze three cases.\\
\emph{Case 1}. If $\dim_\mathbb{C}(Y)=1$, then $y\in f(\Omega_L)=f(\Omega_L\setminus K_n)$, and so $y=f(b_n)$ for some $b_n\in\Omega_L\setminus K_n$. In $L$, the bounded sequence $(b_n)_{n\in\mathbb{N}}\subset\Omega_L$ has a subsequence which converges to some $b\in\partial\Omega_L\subset\partial_2\Omega\cap L$. Since $f$ is $2$-continuous, a passage to the limit yields $y=f(b)\in f(\partial_2\Omega)$. Hence $f(\Omega)\subset f(\partial_2\Omega)$.\\
\emph{Case 2}. For $Y$ of arbitrary dimension, by the already proved inclusion we deduce that $\varphi(f(\Omega))\subset\varphi(f(\partial_2\Omega))$ for every $\varphi\in Y^*$. As in the proof of Proposition~\ref{p.range}(a) (Step 3) we conclude that $f(\Omega)\subset\overline{\mathrm{co}}(f(\partial_2\Omega))$.\\
\emph{Case 3}. If $m:=\dim_\mathbb{C}(Y)\in\mathbb{N}^*\setminus\{1\}$, then $y\in f(\Omega_L)\subset\mathrm{co}(f(\Omega_L\setminus K_n))$, and so $y=\sum_{i=0}^m\lambda_{i,n}f(b_{i,n})$ for some $b_{i,n}\in\Omega_L\setminus K_n$ and $\lambda_{i,n}\in[0,1]$ ($0\le i\le m$), such that $\sum_{i=0}^m\lambda_{i,n}=1$. By taking a convergent subsequence in $L^{m+1}\times[0,1]^{m+1}$, we may assume that $\lim_{n\to\infty}b_{i,n}=b_i\in\partial_2\Omega\cap L$ in $L$ and $\lim_{n\to\infty}\lambda_{i,n}=\lambda_i\in[0,1]$ for $0\le i\le m$. Since $f$ is $2$-continuous, it follows that $y=\sum_{i=0}^m\lambda_if(b_i)\in\mathrm{co}(f(\partial_2\Omega))$. Hence $f(\Omega)\subset\mathrm{co}(f(\partial_2\Omega))$. By the above cases we conclude that $f(\Omega)\sqsubset f(\partial_2\Omega)$.\\
To show the equality, let a continuous seminorm $p:Y\rightarrow\mathbb{R}_+$. By the range inclusion we see that $\sup p(f(\Omega))\le\sup p(f(\partial_2\Omega))$. Now let us fix $b\in\partial_2\Omega$. Hence $b$ is the limit in some $L\in\Gamma_2(X)$ of a sequence $(b_n)_{n\in\mathbb{N}}\subset\Omega_L$. Since $f|_{\Omega_L\cup\partial\Omega_L}$ is continuous, it follows that $p(f(b))=\lim_{n\to\infty}p(f(b_n))\le\sup p(f(\Omega_L))\le\sup p(f(\Omega))$.\\
In order to prove the last implication, assume $f(\partial_2\Omega)\subset D$, that is, $\partial_2\Omega\subset f^{-1}(D)$. Set $K:=\Omega\setminus f^{-1}(D)$. For every $L\in\Gamma_2(X)$ such that $\Omega_L$ is bounded, the $2$-cut $K_L=\Omega_L\setminus f^{-1}(D)=\overline\Omega_L\setminus f^{-1}(D)$ is compact in $L$. It follows that $K$ satisfies the condition from Proposition~\ref{p.excision}. Since $\Omega\setminus K\subset f^{-1}(D)$, by the equivalence from Proposition~\ref{p.excision} we conclude that $f(\Omega)\subset D$.
\end{proof}

The following version of Theorem~\ref{t.boundary} is closer to a maximum modulus principle.

\begin{corollary}[boundary principle]\label{c.boundary}
Assume $X$ is a Hausdorff topological vector space. Let a $2$-bounded open set $\Omega\subset X$ and a map $f\in\mathcal{H}_\mathrm{G}(\Omega,Y)\cap\mathcal{C}(\overline\Omega,Y)$. Then $f(\Omega)\sqsubset f(\partial\Omega)$. In particular, for every continuous seminorm $p:Y\rightarrow\mathbb{R}_+$, we have $\sup_{x\in\Omega}p(f(x))=\sup_{x\in\partial\Omega}p(f(x))$.
For every domain of holomorphy $D\subset Y$,
\[f(\overline\Omega)\subset D\iff f(\partial\Omega)\subset D.\]
If $Y=\mathbb{C}$, then $f(\overline\Omega)=f(\partial\Omega)$.
\end{corollary}

\begin{proof}
Since $\Omega$ is $2$-open and $2$-bounded, $\partial_2\Omega\subset\partial\Omega$, and $f$ is $2$-continuous, the conclusion follows by Theorem~\ref{t.boundary}.
\end{proof}

\begin{remark}\label{r.boundary}
If $Y$ a unital complex Banach algebra, in the above results we may consider the set $D=\mathrm{U}(Y)$ of all invertible elements of $Y$.
\end{remark}

\begin{example}[$2$-boundary principle]\label{ex.boundary}
Let us consider $\alpha>1$, a normed space $X$ with $\dim_\mathbb{C}(X)\ge3$, a nonzero linear functional $\varphi:X\rightarrow\mathbb{C}$, and the sets
\[\Omega:=\{x\in X\,|\,\|x\|<|\varphi(x)|^\alpha\},\qquad F:=\{x\in X\,|\,\|x\|=|\varphi(x)|^\alpha\}.\]
Then $\Omega$ is finitely open and unbounded, $\partial_2\Omega\subset F$, and
\[f(\Omega)\subset f(F),\quad\mbox{for every }f\in\mathcal{H}_\mathrm{G}(\Omega)\cap\mathcal{C}(\Omega\cup F).\]
\end{example}

\subsection{Vector-valued Hartogs companions in arbitrary dimension.}\label{ss.Hc}

\begin{remark}[Hartogs 1-companion]\label{r.dim=1}
For nonempty open set $\Omega\subset\mathbb{C}$, compact subset $K\subset\Omega$, and map $f\in\mathcal{H}(\Omega\setminus K,Y)$, we may consider the \emph{Hartogs 1-companion} $\tilde f\in\mathcal{H}(\Omega,Y)$ defined by (\ref{e.companion1}) or by (\ref{e.representation}) ($Y$-valued integrals). A similar construction is possible if we replace $\mathbb{C}$ by any complex line $L\in\Gamma_1(X)$.
\end{remark}

For the construction of the Hartogs companion in arbitrary dimension we use a slicing technique with linear varieties of finite dimension. More precisely, for fixed sets $K\subset\Omega\subset X$ and map $f:\Omega\setminus K\rightarrow Y$, and arbitrary $L\in\Gamma_{2,\infty}(X)$, we have $K_L\subset\Omega_L\subset L$. If $\Omega_L$ is open, $K_L$ is compact, and $f_L:=f|_{\Omega_L\setminus K_L}\in\mathcal{H}(\Omega_L\setminus K_L,Y)$, then Lemma~\ref{l.Y-companion}(a) provides a unique Hartogs companion $\tilde f_L\in\mathcal{H}(\Omega_L,Y)$. Then we will show that any two Hartogs companions $\tilde f_{L_1}$ and $\tilde f_{L_2}$ agree on $\Omega_{L_1}\cap\Omega_{L_2}$, and so all $\tilde f_L$ may be patched together to define a map $\tilde f\in\mathcal{H}_\mathrm{G}(\Omega,Y)$. The coincidence set will be the union of all coincidence sets of the inclusions $K_L\subset\Omega_L$.

\begin{notation}\label{n.C}
For every inclusion $M\subset\Omega_0\subset L\in\Gamma_{2,\infty}(X)$, with $\Omega_0$ open and $M$ compact in $L$, as in Theorem~\ref{t.C-companion}(a) we may consider the \emph{coincidence set}
\[C_{M,\Omega_0}:=\bigcup_{\omega\in\Upsilon_{\Omega_0}}(\omega\cap M_\omega^\mathrm{u})\subset\Omega_0\setminus M,\]
where $\Upsilon_{\Omega_0}$ stands for the set of all components of $\Omega_0$ in $L$ (endowed with its natural topology) and $M_\omega^\mathrm{u}$ denotes the unbounded component of $L\setminus(M\cap\omega)$. By Lemma~\ref{l.Y-companion}(a) we deduce that every map $h\in\mathcal{H}(\Omega_0\setminus M,Y)$ has a unique Hartogs companion $\tilde h\in\mathcal{H}(\Omega_0,Y)$, and that $\tilde h|_{C_{M,\Omega_0}}=h|_{C_{M,\Omega_0}}$.
\end{notation}

\begin{theorem}[Hartogs companion in arbitrary dimension]\label{t.Y-companion}
Let a fixed integer $d\ge2$, a $d$-open set $\Omega\subset X$, and a $d$-compact subset $K\subset\Omega$. Let us define the $d$-coincidence set of the inclusion $K\subset\Omega$ as
\begin{equation}\label{e.C}C_{K,\Omega}^d:=\bigcup_{L\in\Gamma_{2,d}(X)}C_{K_L,\Omega_L}\subset\Omega\setminus K.\end{equation}
Let a map $f\in\mathcal{H}_\mathrm{G}(\Omega\setminus K,Y)$. Then
\begin{description}
\item[(a)] There exists a unique map $\tilde f\in\mathcal{H}_\mathrm{G}(\Omega,Y)$ (which will be called the Hartogs companion\footnote{Applying the theorem for another integer $s\ge2$, $s<d$ leads to the same map $\tilde f$.} of $f$), such that
\begin{equation}\label{e.Y-companion}\tilde f|_{C_{K,\Omega}^d}=f|_{C_{K,\Omega}^d}.\end{equation}
For every $L\in\Gamma_{2,d}(X)$, the restriction $\tilde f|_{\Omega_L}$ is the Hartogs companion of $f|_{\Omega_L\setminus K_L}$. Furthermore, $\tilde f(\Omega)\sqsubset f(\Omega\setminus K)$. For every domain of holomorphy $D\subset Y$, the implication \textup{(\ref{e,inertia})} holds.
\item[(b)] Let a linear variety $E\subset X$ of dimension (possibly infinite) at least $2$. If $K_E\subset K_0\subset\Omega_0\subset\Omega_E$, with $K_0$ $2$-compact and $\Omega_0$ $2$-open in $E$, then $\tilde f|_{\Omega_0}$ is the Hartogs companion of $f|_{\Omega_0\setminus K_0}$.
\item[(c)] For arbitrarily fixed $a\in\Omega$ and $u\in X\setminus\{0\}$, and $(a,u)$-admissible $G\subset\mathbb{C}$, the set $\Omega_{G,u}:=\{x\in\Omega\,|\,C_{G,u}(x)\mbox{ holds}\}$ is $(d-1)$-open and $a\in\Omega_{G,u}$. We have $\Omega_{G,u}+\partial G\cdot u\subset\Omega\setminus K$, and the representation formula \textup{(\ref{e.representation})} holds.
\item[(d)] If $\Omega\setminus K$ is polygonally connected, then $\tilde f|_{\Omega\setminus K}=f$.
\end{description}
If $\Omega$ is finitely open and $K$ is finitely compact, the theorem holds with a few changes: $\Gamma_{2,d}(X)$ is replaced by $\Gamma_{2,\infty}(X)$, the coincidence set from \textup{(\ref{e.C})} is defined by
\[C_{K,\Omega}:=\bigcup_{L\in\Gamma_{2,\infty}(X)}C_{K_L,\Omega_L}=\bigcup_{d\ge2}C_{K,\Omega}^d,\]
the set $\Omega_{G,u}$ from \textup{(c)} is finitely open, and the statement \textup{(d)} is replaced by
\begin{description}
\item[(d')] If $X\setminus K$ is polygonally connected, then $C_{K,\Omega}=\Omega\setminus K$ and $\tilde f|_{\Omega\setminus K}=f$.
\end{description}
\end{theorem}

\begin{proof}
For every $L\in\Gamma_{2,d}(X)$, the open subset $\Omega_L$ of $L$, contains the compact $K_L$, and so the map $f_L:=f|_{\Omega_L\setminus K_L}\in\mathcal{H}(\Omega_L\setminus K_L,Y)$ has a unique Hartogs companion $\tilde f_L\in\mathcal{H}(\Omega_L,Y)$ as in Lemma~\ref{l.Y-companion}(a). If $\Omega_L=\emptyset$, both $f$ and $\tilde f$ are empty maps. To shorten notation, we write $C_{K_L,\Omega_L}$ simply as $C_L$.\\
(a). \emph{The uniqueness of $\tilde f$}. If (\ref{e.Y-companion}) holds for some map $\tilde f\in\mathcal{H}_\mathrm{G}(\Omega,Y)$, then for every $L\in\Gamma_{2,d}(X)$ we have $C_L\subset C_{K,\Omega}^d\cap\Omega_L$, and so $\tilde f|_{C_L}=f|_{C_L}=f_L|_{C_L}=\tilde f_L|_{C_L}$. By the uniqueness of the Hartogs companions $\tilde f_L$, it follows that
\begin{equation}\label{e.L-companion}\tilde f|_{\Omega_L}=\tilde f_L,\quad\mbox{for every }L\in\Gamma_{2,d}(X).\end{equation}
Since $\Omega=\bigcup_{L\in\Gamma_2(X)}\Omega_L$, the above condition yields the uniqueness of $\tilde f$.\\
\emph{The existence of $\tilde f$}. Let us show that a map $\tilde f\in\mathcal{H}_\mathrm{G}(\Omega,Y)$ can be defined by (\ref{e.L-companion}). We claim that for all varieties $L_1,L_2\in\Gamma_{2,d}(X)$ such that $L:=L_1\cap L_2\ne\emptyset$, we have $\tilde f_{L_1}|_{\Omega\cap L}=\tilde f_{L_2}|_{\Omega\cap L}$ (the Hartogs companions ``agree''). In order to prove this, let us fix such $L_1,L_2,L$, together with $a\in\Omega\cap L$. We next analyze two cases.\\
\emph{Case 1}. If $L\ne\{a\}$, then $L\in\Gamma_{1,d}(X)$ and $a\in\Omega_L$. By the representation (\ref{e.representation}) from Lemma~\ref{l.Y-companion}(c) we get $\tilde f_{L_1}(a)=\tilde f_{L_2}(a)$ (both are uniquely determined by $f|_{\Omega_{L}\setminus K}$).\\
\emph{Case 2}. If $L=\{a\}$, then $L_1\cap L_3\ne\{a\}\ne L_2\cap L_3$ for some $L_3\in\Gamma_2(X)$, such that $a\in L_3$. By the conclusion of the first case we deduce that $\tilde f_{L_1}(a)=\tilde f_{L_3}(a)=\tilde f_{L_2}(a)$. Our claim is proved. Consequently, there exists a unique map $\tilde f\in\mathcal{H}_\mathrm{G}(\Omega,Y)$ defined by (\ref{e.L-companion}). This map satisfies (\ref{e.Y-companion}), since for every $L\in\Gamma_{2,d}(X)$ we see that $\tilde f|_{C_L}=\tilde f_L|_{C_L}=f_L|_{C_L}=f|_{C_L}$. We thus have proved the existence and uniqueness of $\tilde f$ satisfying (\ref{e.Y-companion}). By (\ref{e.Y-range}) and (\ref{e.L-companion}) it follows that
\[\tilde f(\Omega)=\bigcup_{L\in\Gamma_{2,d}(X)}\tilde f_L(\Omega_L)\sqsubset\bigcup_{L\in\Gamma_{2,d}(X)}f_L(\Omega_L\setminus K_L)=f(\Omega\setminus K).\]
If $D\subset Y$ is a domain of holomorphy and $f(\Omega\setminus K)\subset D$, then by Lemma~\ref{l.Y-companion}(a) we deduce that $\tilde f_L(\Omega_L)\subset D$ for every $L\in\Gamma_{2,d}(X)$, which yields $\tilde f(\Omega)\subset D$.\\
(c). Let us fix $a\in\Omega$ and $u\in X\setminus\{0\}$, and an $(a,u)$-admissible set $G\subset\mathbb{C}$. As in the proof of Theorem~\ref{t.C-companion}(c) we see that $\Omega_{G,u}=A^\mathrm{c}\cap B^\mathrm{c}$, with $A:=\Omega^\mathrm{c}-\overline G\cdot u$ and $B:=K-G^\mathrm{c}\cdot u$ (where $G^\mathrm{c}=\mathbb{C}\setminus G$ and the other complements are considered in $X$). We next show that $A$ and $B$ are $(d-1)$-closed. To this aim, fix $L\in\Gamma_{1,d-1}(X)$. Clearly, $L+\mathbb{C}\cdot u\subset L'$ for some $L'\in\Gamma_{1,d}(X)$. It is easy to check that
\begin{eqnarray*}
A\cap L\!\!\!\!&=&\!\!\!\!(\Omega^\mathrm{c}-\overline G\cdot u)\cap L=(\Omega^\mathrm{c}\cap L'-\overline G\cdot u)\cap L,\\
B\cap L\!\!\!\!&=&\!\!\!\!(K-G^\mathrm{c}\cdot u)\cap L=(K\cap L'-G^\mathrm{c}\cdot u)\cap L.
\end{eqnarray*}
In the vector space $\mathrm{span}(L')$, all four sets $\overline G\cdot u$, $K\cap L'$, $\Omega^\mathrm{c}\cap L'$, $G^\mathrm{c}\cdot u$ are closed, and the first two are compact. Therefore, $\Omega^\mathrm{c}\cap L'-\overline G\cdot u$ and $K\cap L'-G^\mathrm{c}\cdot u$ are closed in both $\mathrm{span}(L')$ and $L'$. Hence $A\cap L$ and $B\cap L$ are closed in $L\subset L'$. As $L$ was arbitrary, we conclude that $A$ and $B$ are $(d-1)$-closed, and hence that $\Omega_{G,u}$ is $(d-1)$-open. The representation formula (\ref{e.representation}) follows at once by Lemma~\ref{l.Y-companion}(c).\\
(b). According to (a) (applied for $d=2$), the map $f_0:=f|_{\Omega_0\setminus K_0}$ has a unique Hartogs companion $\tilde f_0\in\mathcal{H}_\mathrm{G}(\Omega_0,Y)$. As in the proof of Theorem~\ref{t.C-companion}(b), by using the representation formula from (c) for both $\tilde f$ and $\tilde f_0$ we deduce that $\tilde f|_{\Omega_0}=\tilde f_0$.\\
(d). Assume $\Omega\setminus K$ is polygonally connected. Let us choose $c\in\Omega\setminus\{0\}$, such that $\delta:=[1,\infty[\cdot c\subset X\setminus K$. Set $C:=\{x\in X\,|\,[c,x]\subset\Omega\setminus K\}$. Thus $c\in C\subset\Omega\setminus K$ and the set $C-c$ is real-absorbing, since $\Omega$ is also $1$-open. We claim that $\tilde f|_C=f|_C$. In order to prove this, let us fix $x\in C$. Clearly, $\Delta:=[c,x]\cup\delta\subset L$ for some $L\in\Gamma_2(X)$. The subset $\Delta\subset L\setminus K$ is unbounded and connected in $L$, and so $\Delta\subset K_L^\mathrm{u}$. Hence $x\in[c,x]\subset\Omega_L\cap K_L^\mathrm{u}\subset C_L$. Since $\tilde f|_{\Omega_L}=\tilde f_L$, it follows that $\tilde f(x)=\tilde f_L(x)=f(x)$. Our claim is proved. As $\Omega\setminus K$ is $2$-open and polygonally connected, by Theorem~\ref{t.identity} we conclude that $\tilde f|_{\Omega\setminus K}=f$.\\
From now on we assume that $\Omega$ is finitely open and $K$ is finitely compact. For every integer $d\ge2$, by applying the already proved result we get the coincidence set $C_d:=C_{K,\Omega}^d$ and a unique map $\tilde f_d\in\mathcal{H}_\mathrm{G}(\Omega,Y)$, such that $\tilde f_d|_{C_d}=f|_{C_d}$. For arbitrary integers $s\ge d\ge2$ we have $C_d\subset C_s$, and so $\tilde f_s|_{C_d}=f|_{C_d}$, which leads by the uniqueness of the Hartogs companion $\tilde f_d$ to $\tilde f_s=\tilde f_d$. For $\tilde f:=\tilde f_2\in\mathcal{H}_\mathrm{G}(\Omega\setminus K,Y)$, we deduce that $\tilde f=\tilde f_d$ and $\tilde f|_{C_d}=f|_{C_d}$ for every $d\ge2$, and so $\tilde f|_{C_{K,\Omega}}=f|_{C_{K,\Omega}}$. It is easily seen that all statements from (a,b,c) hold for every $d\ge2$. Therefore, the set $\Omega_{G,u}$ from (c) is finitely open.\\
(d'). Let us fix $x\in\Omega\setminus K$. There exists $c\in\Omega\setminus\{0\}$, such that $\delta:=[1,\infty[\cdot c\subset X\setminus K$. As $X\setminus K$ is polygonally connected, there is a polygonal chain $\Lambda\subset X\setminus K$ joining $x$ to $c$. Clearly, $\Delta:=\Lambda\cup\delta\subset L$ for some $L\in\Gamma_{2,\infty}(X)$. The subset $\Delta\subset L\setminus K$ is unbounded and connected in $L$, and so $\Delta\subset K_L^\mathrm{u}$. Hence $x\in\Omega_L\cap K_L^\mathrm{u}\subset C_L$. We thus have proved the claimed set equality. It follows that $\tilde f|_{\Omega\setminus K}=f$.
\end{proof}

We next show a weaker version of the Hartogs Kugelsatz in arbitrary dimension. For $X=\mathbb{C}^n$ and $Y=\mathbb{C}$, this version is equivalent to Theorem~\ref{t.Kugelsatz}.

\begin{corollary}[Kugelsatz]\label{c.Kugelsatz}
Let a finitely open set $\Omega\subset X$ and a finitely compact subset $K\subset\Omega$. The \emph{Hartogs companion operator}
\[H_Y:\mathcal{H}_\mathrm{G}(\Omega\setminus K,Y)_\mathrm{k}\rightarrow\mathcal{H}_\mathrm{G}(\Omega,Y)_\mathrm{k}\,,\qquad H_Y(f)=\tilde f,\]
is linear, continuous, and surjective. A right inverse of $H_Y$ is the restriction operator
\[\rho:\mathcal{H}_\mathrm{G}(\Omega,Y)_\mathrm{k}\rightarrow\mathcal{H}_\mathrm{G}(\Omega\setminus K,Y)_\mathrm{k}\,,\qquad\rho(g)=g|_{\Omega\setminus K}.\]
If $X\setminus K$ is polygonally connected, then $H_Y$ and $\rho$ are isomorphisms of locally convex spaces.
\end{corollary}

\begin{proof}
The linearity of $H:=H_Y$ is immediate. By Theorem~\ref{t.Y-companion}(a) we deduce that $H(\rho(g))=H(g|_{\Omega\setminus K})=g$ for every $g\in\mathcal{H}_\mathrm{G}(\Omega,Y)$. Hence $\rho$ is a (continuous) right inverse of $H$, which is surjective. In order to show that $H$ is continuous, let us fix a seminorm $p_M$ as in Definition~\ref{d.G-holomorphy}(e), for $L\in\Gamma_{2,\infty}(X)$, compact subset $M\subset\Omega_L$, and continuous seminorm $p:Y\rightarrow\mathbb{R}_+$. In $L$ let us choose the sets $K_0\subset\Omega_0$, with $K_0$ compact and $\Omega_0$ open, such that $K_L\subset\mathring K_0$ and $M\subset\Omega_0$, and $\overline\Omega_0\subset\Omega_L$ is compact. According to Theorem~\ref{t.Y-companion}(b), for every $f\in\mathcal{H}_\mathrm{G}(\Omega\setminus K,Y)$,  the map $\tilde f|_{\Omega_0}$ is the Hartogs companion of $f|_{\Omega_0\setminus K_0}$. By (\ref{e.Y-range}) we get successively
\[\tilde f(M)\subset\tilde f(\Omega_0)\sqsubset f(\Omega_0\setminus K_0)\subset f(\overline\Omega_0\setminus\mathring K_0),\qquad\tilde f(M)\subset\overline{\mathrm{co}}(f(\overline\Omega_0\setminus\mathring K_0)).\]
Consequently, $p_M(\tilde f)\le p_{\,\overline\Omega_0\setminus\mathring K_0}(f)$ for every $f\in\mathcal{H}_\mathrm{G}(\Omega\setminus K,Y)$. Therefore, $H$ is continuous. If $X\setminus K$ is polygonally connected, by Theorem~\ref{t.Y-companion}(d') it follows that $\rho(H(f))=\tilde f|_{\Omega\setminus K}=f$ for every $\mathcal{H}_\mathrm{G}(\Omega\setminus K,Y)$. We thus conclude that both $H$ and $\rho$ are continuous isomorphisms of locally convex spaces.
\end{proof}

\begin{remark}[Kugelsatz]\label{r.Kugelsatz}
The above corollary still holds for $d$-open $\Omega$ and $d$-compact $K$ (with $d\ge2$), the operator $H_Y:\mathcal{H}_\mathrm{G}(\Omega\setminus K,Y)_{\mathrm{k}(d)}\rightarrow\mathcal{H}_\mathrm{G}(\Omega,Y)_{\mathrm{k}(d)}$, and with the polygonal connectedness of $\Omega\setminus K$ instead of that of $X\setminus K$.
\end{remark}

\noindent Indeed, the proof remains valid if we consider $L\in\Gamma_{2,d}(X)$ (instead of $\Gamma_{2,\infty}(X)$) and we use Theorem~\ref{t.Y-companion}(d) (instead of (d')) for the connectedness of $\Omega\setminus K$.

\begin{proposition}[multiplication and composition property of Hartogs companions]\label{p.Y-composition}
Let $d\ge2$, a $d$-open set $\Omega\subset X$, and a $d$-compact subset $K\subset\Omega$. Then
\begin{description}
\item[(a)] We have $\widetilde{(\alpha f)}=\tilde\alpha\tilde f$ for all $\alpha\in\mathcal{H}_\mathrm{G}(\Omega\setminus K)$ and $f\in\mathcal{H}_\mathrm{G}(\Omega\setminus K,Y)$.
\item[(b)] Let a domain of holomorphy $D\subset Y$, a sequentially complete complex Hausdorff locally convex space $Z$, and a map $g\in\mathcal{H}(D,Z)$. Then
\[\widetilde{(g\circ f)}=g\circ\tilde f,\quad\mbox{for every }f\in\mathcal{H}_\mathrm{G}(\Omega\setminus K,Y)\mbox{ with }f(\Omega\setminus K)\subset D.\]
\end{description}
\end{proposition}

\begin{proof}
(a). For $\alpha,f$ as in (a), $\alpha f\in\mathcal{H}_\mathrm{G}(\Omega\setminus K,Y)$. Since $(\tilde\alpha\tilde f)|_{C_{K,\Omega}^d}=(\alpha f)|_{C_{K,\Omega}^d}$, by the uniqueness of the Hartogs companion we conclude that $\widetilde{(\alpha f)}=\tilde\alpha\tilde f$.\\
(b). Let $f\in\mathcal{H}_\mathrm{G}(\Omega\setminus K,Y)$ as in (b). Hence $\tilde f(\Omega)\subset D$, by Theorem~\ref{t.Y-companion}(a). We have $g\circ f\in\mathcal{H}_\mathrm{G}(\Omega\setminus K,Z)$ and $g\circ\tilde f\in\mathcal{H}_\mathrm{G}(\Omega,Z)$ (the composition property from Herv\'e\,\cite{herve}, Th.\,3.1.10, p.57). Since $(g\circ\tilde f)|_{C_{K,\Omega}^d}=(g\circ f)|_{C_{K,\Omega}^d}$, by the uniqueness of the Hartogs companion we conclude that $\widetilde{(g\circ f)}=g\circ\tilde f$.
\end{proof}

\section{Regularity of Hartogs companions}\label{s.regularity}

Regularity results for Hartogs companions may be obtained by showing that the representation (\ref{e.representation}) from Theorem~\ref{t.Y-companion}(c) holds locally, on neighborhoods of points. Let us note that under the hypothesis of Theorem~\ref{t.regularity} below, the sets $\Omega_{G,u}$ from Theorem~\ref{t.Y-companion}(c) are not necessarily open, unless $K$ is bounded.

Let us recall four holomorphy types stronger than G\^ateaux holomorphy:

\begin{definition}[holomorphy/analyticity types]\label{d.holomorphy}
Assume $X$ is a Hausdorff locally convex space. A map $f\in\mathcal{H}_\mathrm{G}(\Omega,Y)$ on an open subset $\Omega\subset X$ is called
\begin{description}
\item[(LB)] \emph{locally bounded holomorphic}, if and only if $f$ is locally bounded.
\item[(FR)] \emph{holomorphic} (or \emph{Fr\'echet analytic}\footnote{For several other conditions equivalent to holomorphy, see Herv\'e\,\cite{herve}, Def.\,3.1.1, p.52.}), if and only if $f$ is continuous.
\item[(HY)] \emph{hypoanalytic}, if and only if $f$ is hypocontinuous (that is, all restrictions $f|_M$ to compact subsets $M\subset\Omega$ are continuous).
\item[(MS)] \emph{Mackey/Silva holomorphic}, if and only if\footnote{This is not the definition, but an equivalent condition (Dineen\,\cite{dineen1}, Prop.\,2.18(a,b), p.61).} for all $a\in\Omega$ and bounded subset $B\subset X$, there exists $\varepsilon>0$, such that $f(\Omega\cap(a+\varepsilon B))$ is bounded.
\end{description}
The following inclusions hold (with the standard notations for the four vector spaces consisting of maps as in (LB)--(MS); see Dineen\,\cite{dineen1}, p.62):
\[\mathcal{H}_\mathrm{LB}(\Omega,Y)\subset\mathcal{H}(\Omega,Y)\subset\mathcal{H}_\mathrm{HY}(\Omega,Y)\subset\mathcal{H}_\mathrm{M}(\Omega,Y)\subset\mathcal{H}_\mathrm{G}(\Omega,Y).\]
\end{definition}

Without assuming that $X$ is locally convex, we next show that the conditions from the above definition are inherited from a map by its Hartogs companion.

\begin{theorem}[regularity]\label{t.regularity}
Assume $X$ is a Hausdorff topological vector space. Let an open set $\Omega\subset X$, a $2$-bounded closed set $K\subset\Omega$, a map $f\in\mathcal{H}_\mathrm{G}(\Omega\setminus K,Y)$, and its Hartogs companion $\tilde f\in\mathcal{H}_\mathrm{G}(\Omega,Y)$ as in Theorem~\ref{t.Y-companion}(a).
\begin{description}
\item[(a)] For arbitrary $a\in\Omega$ and $u\in X\setminus\{0\}$, and $(a,u)$-admissible set $G\subset\mathbb{C}$, let $D_{G,u}$ denote the component of $a$ in the open set $(\Omega^\mathrm{c}-\overline G\cdot u)^\mathrm{c}\cap(K-\partial G\cdot u)^\mathrm{c}$ (hence $D_{G,u}$ is an open neighborhood of $a$). Then $D_{G,u}+\partial G\cdot u\subset\Omega\setminus K$ and the representation formula \textup{(\ref{e.representation})} holds for every $x\in D_{G,u}$.
\item[(lb)] If $f$ is locally bounded, then so is $\tilde f$. If $p:Y\rightarrow\mathbb{R}_+$ is a continuous seminorm and $p\circ f$ is locally bounded, then so is $p\circ\tilde f$.
\item[(fr)] If $f$ is continuous, then so is $\tilde f$.
\item[(hy)] If $f$ is hypocontinuous, then so is $\tilde f$.
\item[(ms)] If $f$ satisfies the condition from Definition~\ref{d.holomorphy}(MS), then so does $\tilde f$.
\end{description}
If $X$ is locally convex, any holomorphy from Definition~\ref{d.holomorphy} is inherited by $\tilde f$ from $f$.
\end{theorem}

\begin{proof}
Let $\mathcal{S}_Y$ denote the set of all continuous seminorms on $Y$.\\
(a). For fixed $a$, $u$, and $G$ as in (a), set $A:=\Omega^\mathrm{c}-\overline G\cdot u$ and $B_1:=K-\partial G\cdot u$, and $D:=A^\mathrm{c}\cap B_1^\mathrm{c}$. Thus $a\in A^\mathrm{c}\subset\Omega$. For every $x\in X$, we have the equivalences
\[x\in D\iff\left\{\begin{array}{l}x\notin\Omega^\mathrm{c}-\overline G\cdot u\\ x\notin K-\partial G\cdot u\end{array}\right.\iff\left\{\begin{array}{l}x+\overline G\cdot u\subset\Omega\\ x+\partial G\cdot u\subset\Omega\setminus K.\end{array}\right.\]
Hence $D+\partial G\cdot u\subset\Omega\setminus K$. Since $G$ is $(a,u)$-admissible, it follows that $a\in D$. All four sets $\overline G\cdot u$, $\partial G\cdot u$, $\Omega^\mathrm{c}$, $K$ are closed, and the first two are compact. Hence both $A$ and $B_1$ are closed, and so $D$ is open. Therefore, $D_{G,u}$ is an open neighborhood of $a$. Let us define the map
\[\bar f:D_{G,u}\rightarrow Y,\qquad\bar f(x)=\frac1{2\pi\mathrm{i}}\int_{\partial G}\frac{f(x+\zeta u)}\zeta\mathrm{d}\zeta.\]
By the representation formula (\ref{e.representation}) we get $\bar f|_C=\tilde f|_C$, where $C:=\Omega_{G,u}\cap D_{G,u}\ni a$. As in the proof of Theorem~\ref{t.outer} (Step 4) we deduce that $\bar f\in\mathcal{H}_\mathrm{G}(D_{G,u},Y)$. According to Theorem~\ref{t.Y-companion}(c), the set $\Omega_{G,u}$ is $1$-open, and hence so is $C$. Since the set $D_{G,u}$ is open and connected, and hence polygonally connected, by Theorem~\ref{t.identity} it follows that $\bar f=\tilde f|_{D_{G,u}}$. We conclude that the representation (\ref{e.representation}) holds for every $x\in D_{G,u}$.\\
(lb). Assume $f$ is locally bounded. For fixed $a\in\Omega$, let us choose $u$ and $G$ as in (a). Since $a+\partial G\cdot u$ is compact and $f$ is locally bounded, by a standard compactness argument we find a neighborhood $U\subset D_{G,u}$ of $a$, such that $P:=f(U+\partial G\cdot u)$ is bounded. We claim that $\tilde f(U)$ is bounded. In order to prove this, let us fix $p\in\mathcal{S}_Y$. As $0\in G$, we have $B_\mathbb{C}(0,r)\subset G$ for some $r>0$. By the representation formula (\ref{e.representation}) on $D_{G,u}\supset U$ it follows that
\[p(\tilde f(x))\le\frac{\ell(\partial G)}{2\pi r}\sup_{\zeta\in\partial G}p(f(x+\zeta u))\le\frac{\ell(\partial G)}{2\pi r}\sup p(P),\quad\mbox{for every }x\in U,\]
where $\ell(\partial G)$ denotes the length of the boundary $\partial G$ (which consists of finitely many piecewise $\mathcal{C}^1$ Jordan curves). We thus conclude that $\tilde f(U)$ is bounded, and hence that $\tilde f$ is locally bounded. The proof of the second statement from (lb) is similar.\\
(fr). In order to show that $\tilde f$ is continuous, let us fix $a\in\Omega$ and $p\in\mathcal{S}_Y$, and $\varepsilon>0$. For $a$, let us choose $u$ and $G$ as in (a), together with $r>0$, such that $B_\mathbb{C}(0,r)\subset G$. By the representation formula (\ref{e.representation}) from (a) it follows that
\[\tilde f(x)-\tilde f(a)=\frac1{2\pi\mathrm{i}}\int_{\partial G}\frac{f(x+\zeta u)-f(a+\zeta u)}\zeta\mathrm{d}\zeta,\quad\mbox{for every }x\in D_{G,u}.\]
Since $a+\partial G\cdot u$ is compact and $f$ is continuous, by a standard compactness argument we find a neighborhood $U_\varepsilon\subset D_{G,u}$ of $a$, such that
\[p(f(x+\zeta u)-f(a+\zeta u))<\frac{r\varepsilon}{\ell(\partial G)},\quad\mbox{for all }x\in U_\varepsilon,\,\zeta\in\partial G.\]
Consequently, for every $x\in U_\varepsilon$ we have
\[p(\tilde f(x)-\tilde f(a))\le\frac{\ell(\partial G)}{2\pi r}\sup_{\zeta\in\partial G}p(f(x+\zeta u)-f(a+\zeta u))\le\frac\varepsilon{2\pi}<\varepsilon.\]
Hence $\tilde f$ is continuous at $a$. It follows that $\tilde f$ is continuous on $\Omega$.\\
(hy). Let us fix a compact subset $M\subset\Omega$ and $a\in M$, together with $p\in\mathcal{S}_Y$ and $\varepsilon>0$. For $a$, let us choose $u$ and $G$ as in (a), and a closed neighborhood $U\subset D_{G,u}$ of $a$. The set $M_0:=U\cap M+\partial G\cdot u\subset\Omega\setminus K$ is compact, and so $f|_{M_0}$ is uniformly continuous. Consequently, there is a balanced neighborhood $V\subset X$ of $0$, such that $W_\varepsilon:=a+V\subset U$ and
\[p(f(y)-f(z))<\frac{r\varepsilon}{\ell(\partial G)},\quad\mbox{for all }y,z\in M_0,\mbox{ with }y-z\in V.\]
For every $x\in W_\varepsilon\cap M$ we see that $x+\partial G\cdot u\subset M_0$ and $x-a\in V$, which lead as in the proof of (fr) to $p(\tilde f(x)-\tilde f(a))\le\frac\varepsilon{2\pi}<\varepsilon$. Hence $\tilde f|_M$ is continuous at $a$. It follows that $\tilde f|_M$ is continuous, and hence that $\tilde f$ is hypocontinuous.\\
(ms). Let us fix $a\in\Omega$ and a nonempty balanced bounded subset $B\subset X$. For $a$, choose $u$ and $G$ as in (a). The balanced set $B_u:=B+B_\mathbb{C}(0,1)\cdot u$ is bounded and $B\subset B_u$. For each $\zeta\in\partial G$ and for $a_\zeta:=a+\zeta u\in\Omega\setminus K$, there is $\varepsilon_\zeta>0$, such that
\[a+\varepsilon_\zeta B_u\subset D_{G,u},\qquad f(a_\zeta+\varepsilon_\zeta B_u)\mbox{ is bounded}.\]
As $\partial G$ is compact, there is a finite subset $F\subset\partial G$, such that $\partial G\subset\bigcup_{\xi\in F}B_\mathbb{C}(\xi,\varepsilon_\xi)$. For $\varepsilon:=\min_{\xi\in F}\varepsilon_\xi>0$, we have $a+\varepsilon B\subset D_{G,u}$ and $P:=\bigcup_{\xi\in F}f(a_\xi+\varepsilon_\xi B_u)$ is bounded. Let us observe that
\begin{eqnarray*}
a+\varepsilon B+\partial G\cdot u\!\!\!\!&\subset&\!\!\!\!\bigcup_{\xi\in F}[a+\varepsilon B+B_\mathbb{C}(\xi,\varepsilon_\xi)\cdot u]=\bigcup_{\xi\in F}[a_\xi+\varepsilon B+B_\mathbb{C}(0,\varepsilon_\xi)\cdot u]\\
&\subset&\!\!\!\!\bigcup_{\xi\in F}[a_\xi+\varepsilon_\xi B+\varepsilon_\xi B_\mathbb{C}(0,1)\cdot u]=\bigcup_{\xi\in F}(a_\xi+\varepsilon_\xi B_u),
\end{eqnarray*}
and hence that $f(a+\varepsilon B+\partial G\cdot u)\subset P$. Since $P$ is bounded, as in the proof of (lb) (with $a+\varepsilon B$ instead of $U$) it follows that $\tilde f(a+\varepsilon B)$ is bounded.
\end{proof}

\section{Hartogs-type extension theorems.}\label{s.extension}

To avoid repetition, let us first note that unless Theorem~\ref{t.outer}, all results from this section hold together with the following:\\[1mm]
\textbf{Additional conclusions. }\emph{For $f$ and its unique extension $\tilde f$ as above, we have the range inclusion $\tilde f(\Omega)\sqsubset f(\Omega\setminus K)$. Furthermore, for every domain of holomorphy $D\subset Y$, we have the equivalence
\[f(\Omega\setminus K)\subset D \ \Longleftrightarrow \ \tilde f(\Omega)\subset D.\]
If $Y=\mathbb{C}$, then $\tilde f(\Omega)=f(\Omega\setminus K)$.}

\begin{remark}[connectedness conditions]\label{r.connectedness}
All three equivalences from Proposition~\ref{p.connectedness} also hold for a finitely open set $\Omega\subset X$ and a finitely compact subset $K\subset\Omega$.
\end{remark}

\noindent Indeed, the proof of Proposition~\ref{p.connectedness} remains valid if we replace $\mathbb{C}^n$ by $X_\mathrm{f}$ and we use Theorems~\ref{t.identity}, \ref{t.Y-companion}(d') instead of the classical identity theorem and Theorem~\ref{t.C-companion}(d).

The next result is the correspondent of Theorem~\ref{t.Hartogs=} in arbitrary dimension.

\begin{theorem}[extension/Kugelsatz]\label{t.extension=}
Let a finitely open set $\Omega\subset X$ and a finitely compact subset $K\subset\Omega$. The following four statements are equivalent (where in \textup{(i')} we consider $X$ equipped with the finite open topology $\tau_\mathrm{f}$).
\begin{description}
\item[(i)] Every map $f\in\mathcal{H}_\mathrm{G}(\Omega\setminus K,Y)$ has a (unique) extension $\tilde f\in\mathcal{H}_\mathrm{G}(\Omega,Y)$.
\item[(i')] Every locally constant map $g:\Omega\setminus K\rightarrow\mathbb{C}$ has an extension $\tilde g\in\mathcal{H}_\mathrm{G}(\Omega)$.
\item[(ii)] The restriction $\rho:\mathcal{H}_\mathrm{G}(\Omega,Y)_\mathrm{k}\rightarrow\mathcal{H}_\mathrm{G}(\Omega\setminus K,Y)_\mathrm{k}$ is an isomorphism of complex vector spaces.
\item[(iii)] $X\setminus K$ is polygonally connected.
\end{description}
In this case, $\rho$ is an isomorphism of locally convex spaces whose inverse is the Hartogs companion operator $H_Y$, and the additional conclusions hold.
\end{theorem}

\begin{proof}
(i)$\Rightarrow$(ii). The extension of every $f\in\mathcal{H}_\mathrm{G}(\Omega\setminus K,Y)$ is the Hartogs companion $H_Y(f)$, and so $\rho(H_Y(f))=H_Y(f)|_{\Omega\setminus K}=f$. Consequently, for $H_Y$ the restriction $\rho$ is a left inverse, but also a right inverse, by Corollary~\ref{c.Kugelsatz}. Since both $\rho$ and $H_Y$ are $\tau_\mathrm{k}$-continuous, $\rho$ is an isomorphism of locally convex spaces.\\
(ii)$\Rightarrow$(i). Since the restriction operator $\rho$ is surjective, (i) holds.\\
(iii)$\Rightarrow$(i) obviously holds by Theorem~\ref{t.Y-companion}(a,d').\\
(i)$\Rightarrow$(i'). Let a locally constant map $g:\Omega\setminus K\rightarrow\mathbb{C}$. Let us choose $y\in Y\setminus\{0\}$ and $\varphi\in Y^*$, such that $\varphi(y)=1$. For $f:=g\cdot y\in\mathcal{H}_\mathrm{G}(\Omega\setminus K,Y)$ and its extension $\tilde f$ as in (i), we have $\tilde g:=\varphi\circ\tilde f\in\mathcal{H}_\mathrm{G}(\Omega)$ and $\tilde g|_{\Omega\setminus K}=\varphi\circ(\tilde f|_{\Omega\setminus K})=\varphi\circ f=g$.\\
(i')$\Rightarrow$(iii). The proof is the same as that of the corresponding implication from Theorem~\ref{t.Hartogs=} (with $\mathbb{C}^n$ replaced by $X_\mathrm{f}$) and uses Remark~\ref{r.connectedness} and Theorem~\ref{t.identity} instead of Proposition~\ref{p.connectedness} and the classical identity theorem.\\
The additional conclusions follow by Theorem~\ref{t.extension=}(a).
\end{proof}

The topological version of the above theorem allows $K$ to have nonempty interior (for instance, when $X$ is normable and $K$ is closed and bounded).

\begin{corollary}[extension/Kugelsatz]\label{c.extension=}
Assume $X$ is a Hausdorff locally convex space. Let an open set $\Omega\subset X$ and a finitely bounded\footnote{Here and in Corollary~\ref{c.2-extension} topological boundedness of $K$ is unnecessary and more restrictive.} closed subset $K\subset\Omega$. The following six statements are equivalent.
\begin{description}
\item[(i)] Every map $f\in\mathcal{H}_\mathrm{G}(\Omega\setminus K,Y)$ has an extension $\tilde f\in\mathcal{H}_\mathrm{G}(\Omega,Y)$.
\item[(ii)] Every map $f\in\mathcal{H}_\mathrm{M}(\Omega\setminus K,Y)$ has an extension $\tilde f\in\mathcal{H}_\mathrm{M}(\Omega,Y)$.
\item[(iii)] Every map $f\in\mathcal{H}_\mathrm{HY}(\Omega\setminus K,Y)$ has an extension $\tilde f\in\mathcal{H}_\mathrm{HY}(\Omega,Y)$.
\item[(iv)] Every map $f\in\mathcal{H}(\Omega\setminus K,Y)$ has an extension $\tilde f\in\mathcal{H}(\Omega,Y)$.
\item[(v)] Every map $f\in\mathcal{H}_\mathrm{LB}(\Omega\setminus K,Y)$ has an extension $\tilde f\in\mathcal{H}_\mathrm{LB}(\Omega,Y)$.
\item[(vi)] $X\setminus K$ is connected.
\end{description}
In this case, all extensions are unique and the additional conclusions hold.
\end{corollary}

\begin{proof}
As $\Omega$ and $X\setminus K$ are finitely open, $K$ is finitely compact, and (vi) is equivalent to the polygonal connectedness of $X\setminus K$, by Theorem~\ref{t.extension=} we see that (i)$\Leftrightarrow$(vi). That (i) yields all (ii)--(v) follows by Theorem~\ref{t.regularity}. Finally, any of (ii)--(v) implies the condition from Theorem~\ref{t.extension=}(i') (locally constant maps $g:\Omega\setminus K\rightarrow\mathbb{C}$ are of all four holomorphy types), which is equivalent to (i), by Theorem~\ref{t.extension=}. We thus have proved the equivalence of all six statements. In (i)--(v) the extension of $f$ is the Hartogs companion, which if unique.
\end{proof}

The following extension theorem only involves $2$-cuts of both sets $K$ and $\Omega$. This result is more general than Theorem~\ref{t.Hartogs} even for $X=\mathbb{C}^n$ and $Y=\mathbb{C}$, since $2$-compact sets may not be closed or bounded and $2$-open sets may not be open.

\begin{theorem}[2-cuts extension]\label{t.2-extension}
Let a $2$-open set $\Omega\subset X$ and a $2$-compact subset $K\subset\Omega$. If $\Omega\setminus K$ is polygonally connected, then every map $f\in\mathcal{H}_\mathrm{G}(\Omega\setminus K,Y)$ has a unique extension $\tilde f\in\mathcal{H}_\mathrm{G}(\Omega,Y)$ and the additional conclusions hold.
\end{theorem}

\begin{proof}
The conclusion is immediate, by Theorem~\ref{t.Y-companion}(a,d) for $d=2$.
\end{proof}

\begin{corollary}[holomorphic extensions]\label{c.2-extension}
Assume $X$ is a Hausdorff locally convex space. Let an open set $\Omega\subset X$ and a $2$-bounded closed subset $K\subset\Omega$, such that $\Omega\setminus K$ is connected. Then the statements (i)--(v) from Corollary~\ref{c.extension=} hold together with the additional conclusions.
\end{corollary}

\begin{proof}
The conclusion follows by applying successively Theorems~\ref{t.2-extension} and~\ref{t.regularity}.
\end{proof}

The above four results deal with inner G\^ateaux holomorphic extensions; we call these ``inner'', since for every variety $L\in\Gamma_2(X)$ the set $K_L\subset\Omega_L$ is compact in $L$ (we may say that $K_L$ is a ``compact hole'' in $\Omega_L$). We next establish a theorem suitable for outer (that is, not inner) extensions. For some fixed $u\in X\setminus\{0\}$, this result only involves particular cuts with linear varieties of the form
\[L_a(u):=a+\mathbb{C}\cdot u,\qquad L_a(u,v):=a+\mathbb{C}\cdot u+\mathbb{C}\cdot v,\]
and with closed \emph{linear $2$-strips} parallel to $u$, defined by
\[L_a(u,v[\varepsilon]):=L_a(u)+\overline B_\mathbb{C}(0,\varepsilon)\cdot v.\]

\begin{theorem}[outer extension]\label{t.outer}
Let the sets $K\subset\Omega\subset X$, such that $\Omega$ and $\Omega\setminus K$ are $2$-open\footnote{Here and in Corollaries~\ref{c.outer}, \ref{c.Hartogs}, the set $K$ may not be $2$-closed or $2$-bounded.} and $\Omega\setminus K$ is polygonally connected. Assume there exists $u\in X\setminus\{0\}$, such that:
\begin{description}
\item[(i)] For all $a\in\Omega$ and $v\in X$, there exists $\varepsilon>0$, such that $K\cap L_a(u,v[\varepsilon])$ is compact in $L_a(u,v)$.
\item[(ii)] $K\cap L_c(u)=\emptyset$ for some $c\in\Omega$.
\end{description}
Then every map $f\in\mathcal{H}_\mathrm{G}(\Omega\setminus K,Y)$ has a unique extension $\tilde f\in\mathcal{H}_\mathrm{G}(\Omega,Y)$. We have $\tilde f(\Omega)\subset\overline{\mathrm{co}}(f(\Omega\setminus K))$. If $Y=\mathbb{C}$, then $\tilde f(\Omega)=f(\Omega\setminus K)$.
\end{theorem}

\begin{proof}
There is no loss of generality in assuming $K\ne\emptyset$. The uniqueness of $\tilde f$ will follow by Theorem~\ref{t.identity}, if we show that $\Omega\setminus K\ne\emptyset$ and $\Omega$ is polygonally connected. We claim that every $a\in\Omega$ can be joined in $\Omega$ by a linear segment to some $a'\in\Omega\setminus K\ne\emptyset$. Clearly, it suffices to prove this for $a\in K$. For such $a$, in $L=L_a(u)\in\Gamma_1(X)$ the set $\Omega_L$ is open and $K_L$ is compact (which follows from (i) for $v=0$), with $a\in K_L\subset\Omega_L$. Therefore, there exists $t>0$, such that $a':=a+tu\notin K_L$ and $[a,a']\subset\Omega_L$. Our claim is proved. Hence for arbitrarily fixed $a,b\in\Omega$, we have $[a,a']\cup[b,b']\subset\Omega$ for some $a',b'\in\Omega\setminus K$. Since $\Omega\setminus K$ is polygonally connected, $a'$ and $b'$ can be joined by a polygonal chain $\Lambda\subset\Omega$, and so $[a,a']\cup\Lambda\cup[b',b]\subset\Omega$. We thus conclude that $\Omega$ is polygonally connected. Now the uniqueness of $\tilde f$ follows by Theorem~\ref{t.identity}. The proof of the existence part is divided into five steps, among which the first four are only based on the condition (i).\\
\emph{Step 1}. We first show the following needed technical property:\\
\emph{For every nonempty closed subset $H\subset\mathbb{C}$ the set $S:=\Omega\setminus(K-H\cdot u)$ is $1$-open.}\\
On the contrary, suppose there exists a $1$-cut $S_L$, which is not open in $L\in\Gamma_1(X)$. Hence there is a sequence $(s_n')_{n\in\mathbb{N}}\subset L\setminus S$, convergent in $L$ to some $s\in S_L\subset\Omega_L$. Thus $L=L_s(v)$ for some $v\in X\setminus\{0\}$. Since $\Omega_L$ is open in $L$ and $s\in\Omega_L$, by the above convergence we may assume $(s_n')_{n\in\mathbb{N}}\subset\Omega_L\setminus S\subset K-H\cdot u$. Consequently, there exist two sequences $(k_n)_{n\in\mathbb{N}}\subset K$ and $(h_n)_{n\in\mathbb{N}}\subset H$, such that
\begin{equation}\label{e.sequence}s_n'=k_n-h_nu,\quad\mbox{for every }n\in\mathbb{N}.\end{equation}
According to (i), there exists $\varepsilon>0$, such that $K_\varepsilon:=K\cap L_s(u,v[\varepsilon)]$ is compact in $L_s(u,v)$. As $\lim_{n\to\infty}s_n'=s$ in $L$, by taking subsequences we may also assume $(s_n')_{n\in\mathbb{N}}\subset s+\overline B_\mathbb{C}(0,\varepsilon)\cdot v$. By (\ref{e.sequence}) we see that $k_n=s_n'+h_nu\in K_\varepsilon$ for every $n\in\mathbb{N}$. Since $K_\varepsilon$ is compact in $L_s(u,v)$, by taking subsequences we may assume the existence of the limit $k:=\lim_{n\to\infty}k_n\in K_\varepsilon\subset K$ in $L_s(u,v)$. By (\ref{e.sequence}) we deduce that $(h_n)_{n\in\mathbb{N}}\subset H$ converges and $h:=\lim_{n\to\infty}h_n\in H$. Now a passage to the limit in $L_s(u,v)$ in the equality (\ref{e.sequence}) forces $s=k-hu\in K-H\cdot u$, which contradicts $s\in S$. We thus have proved the claimed property.\\
\emph{Step 2}. We next show that all sets $\Omega_{G,u}$ defined as in Theorem~\ref{t.Y-companion}(c) are $1$-open. Let us fix $a\in\Omega$ and an $(a,u)$-admissible set $G\subset\mathbb{C}$. Such sets $G$ indeed exist, since in $L=L_a(u)\in\Gamma_1(X)$ the set $K_L\subset\Omega_L$ is compact and $\Omega_L\ne\emptyset$ is open. As in the proof of Theorem~\ref{t.Y-companion}(c) we see that $\Omega_{G,u}=A^\mathrm{c}\cap(\Omega\setminus B)$, where $A=\Omega^\mathrm{c}-\overline G\cdot u$ and $B=K-G^\mathrm{c}\cdot u$, and $A^\mathrm{c}$ is $1$-open. As $G^\mathrm{c}$ is closed in $\mathbb{C}$, according to the property from Step 1, $\Omega\setminus B$ is also $1$-open. We thus conclude that $\Omega_{G,u}$ is $1$-open.\\
\emph{Step 3}. We next define the map $\tilde f$. As the set $\Gamma_1(\Omega,u):=\{L_a(u)\,|\,a\in\Omega\}\subset\Gamma_1(X)$ consists of mutually disjoint complex lines, $(\Omega_L)_{L\in\Gamma_1(\Omega,u)}$ is a partition of $\Omega$. For every $L\in\Gamma_1(\Omega,u)$, the open set $\Omega_L\ne\emptyset$ contains the compact $K_L$. Consequently, the restriction $f_L:=f|_{\Omega_L\setminus K_L}\in\mathcal{H}(\Omega_L\setminus K_L,Y)$ has a unique Hartogs 1-companion $\tilde f_L\in\mathcal{H}(\Omega_L,Y)$ as in Remark~\ref{r.dim=1}. Let us define the map
\[\tilde f:\Omega\rightarrow Y,\qquad\tilde f|_{\Omega_L}=\tilde f_L\quad\mbox{for every }L\in\Gamma_1(\Omega,u).\]
Hence for arbitrarily fixed $a\in\Omega$ and $(a,u)$-admissible set $G\subset\mathbb{C}$, the restriction $\tilde f|_{\Omega_{G,u}}$ may be represented by (\ref{e.representation}).\\
\emph{Step 4}. We next show that $\tilde f\in\mathcal{H}_\mathrm{G}(\Omega,Y)$. To this aim, let us fix $a\in\Omega$, together with $v\in X\setminus\{0\}$ and $\varphi\in Y^*$. Choose an $(a,u)$-admissible set $G\subset\mathbb{C}$. Since $\Omega_{G,u}$ is $1$-open and $a\in\Omega_{G,u}$, there exists $r>0$, such that $a+B_\mathbb{C}(0,r)\cdot v\subset\Omega_{G,u}$. Hence $a+B_\mathbb{C}(0,r)\cdot v+\partial G\cdot u\subset\Omega\setminus K$. As $\Omega\setminus K$ is $2$-open, the set
\[D:=\{(\lambda,\zeta)\in\mathbb{C}^2\,|\,a+\lambda v+\zeta u\in\Omega\setminus K\}\]
is open in $\mathbb{C}^2$ and $B_\mathbb{C}(0,r)\times\partial G\subset D$. Let us define the map
\[F\in\mathcal{H}(D),\qquad F(\lambda,\zeta)=(\varphi\circ f)(a+\lambda v+\zeta u).\]
By the integral representation (\ref{e.representation}) of $\tilde f|_{\Omega_{G,u}}$ it follows that
\[(\varphi\circ\tilde f)(a+\lambda v)=\frac1{2\pi\mathrm{i}}\int_{\partial G}\frac{(\varphi\circ f)(a+\lambda v+\zeta u)}\zeta\mathrm{d}\zeta=\frac1{2
\pi\mathrm{i}}\int_{\partial G}\frac{F(\lambda,\zeta)}\zeta\mathrm{d}\zeta,\]
for every $\lambda\in B_\mathbb{C}(0,r)$. Since differentiation under the integral sign with respect to $\lambda$ holds (and $a,v,\varphi$ were arbitrarily fixed), we conclude that $\tilde f\in\mathcal{H}_\mathrm{G}(\Omega,Y)$.\\
\emph{Step 5}. We finally show that $\tilde f|_{\Omega\setminus K}=f$. Choose $c\in\Omega$ as in (ii) and set
\[C:=\{x\in\Omega\,|\,K\cap L_x(u)=\emptyset\}.\]
It is easy to check that $c\in C$ and $C=\Omega\setminus(K-\mathbb{C}\cdot u)\subset\Omega\setminus K$. By the property from Step 1 we deduce that $C$ is $1$-open, and so $C-c$ is a real-absorbing set. In order to prove that $\tilde f|_C=f|_C$, let us fix $x\in C$. For $L=L_x(u)\in\Gamma_1(\Omega,u)$ we have $K_L=\emptyset$, and so $\tilde f_L=f_L$. By the definition of $\tilde f$ it follows that $\tilde f|_{\Omega_L}=f_L=f|_{\Omega_L}$, which yields $\tilde f(x)=f(x)$. Hence $\tilde f|_C=f|_C$. Since $\Omega\setminus K$ is $2$-open and polygonally connected, by Theorem~\ref{t.identity} we conclude that $\tilde f|_{\Omega\setminus K}=f$. The proof of the existence and uniqueness of the extension $\tilde f$ is now complete. Since every $f\in\mathcal{H}_\mathrm{G}(\Omega\setminus K,Y)$ has a unique extension $\tilde f$, the statements on $\tilde f(\Omega)$ (the general range inclusion and the equality for $Y=\mathbb{C}$) follow by Proposition~\ref{p.range}(a).
\end{proof}

The following corollary strengthens Theorem~\ref{t.2-extension} by weakening its compactness assumption on $K$ (only $2$-cuts parallel to a given vector need to be compact).

\begin{corollary}[outer extension]\label{c.outer}
Let the sets $K\subset\Omega\subset X$, such that $\Omega$ and $\Omega\setminus K$ are $2$-open and $\Omega\setminus K$ is polygonally connected. Assume there exists $u\in X\setminus\{0\}$, such that the cut $K\cap L_a(u,v)$ is compact for all $a\in\Omega$ and $v\in X$. Then every map $f\in\mathcal{H}_\mathrm{G}(\Omega\setminus K,Y)$ has a unique extension $\tilde f\in\mathcal{H}_\mathrm{G}(\Omega,Y)$ and the additional conclusions hold.
\end{corollary}

\begin{proof}
Since conditions (i,ii) from Theorem~\ref{t.outer} are fulfilled, every $f\in\mathcal{H}_\mathrm{G}(\Omega\setminus K,Y)$ has a unique extension $\tilde f\in\mathcal{H}_\mathrm{G}(\Omega,Y)$. We next prove the additional conclusions.\\
\emph{The range inclusion}. For every fixed $a\in\Omega$, there exists $L\in\Gamma_2(X)$, such that $a\in L$ and $L\parallel u$. Since $K_L$ is compact and $\Omega_L$ is open in $L$, the restriction $\tilde f|_{\Omega_L}$ is the Hartogs companion of $f|_{\Omega_L\setminus K_L}$. By (\ref{e.Y-range}) it follows that
\[\tilde f(a)\in\tilde f(\Omega_L)\sqsubset f(\Omega_L\setminus K_L)\subset f(\Omega\setminus K).\]
As $a$ was arbitrary, we conclude that $\tilde f(\Omega)\sqsubset f(\Omega\setminus K)$.\\
\emph{The equivalence for the domain of holomorphy $D\subset Y$}. Assume $f(\Omega\setminus K)\subset D$. For every fixed $a\in\Omega$ and for $L\in\Gamma_2(X)$ as above, we have $f(\Omega_L\setminus K_L)\subset D$, and so $\tilde f(a)\in\tilde f(\Omega_L)\subset D$, by Theorem~\ref{t.inertia}. We thus conclude that $\tilde f(\Omega)\subset D$.
\end{proof}

For some inclusions $K\subset\Omega\subset X$ Theorem~\ref{t.outer} applies, while Corollary~\ref{c.outer} fails:

\begin{example}[outer extension]\label{ex.Y-Hartogs}
Assume $\dim_\mathbb{C}(X)\ge3$.
\begin{description}
\item[(a)] For $e\in X\setminus\{0\}$ and $3$-compact subset $B\subset X$, let $K:=B+\mathbb{R}_+\!\cdot e$. Then every map $f\in\mathcal{H}_\mathrm{G}(X\setminus K,Y)$ has a unique extension $\tilde f\in\mathcal{H}_\mathrm{G}(X,Y)$.
\item[(b)] Assume $X$ is a normed space. Let a linear functional $\varphi:X\rightarrow\mathbb{C}$, an upper semicontinuous map $r:\mathbb{C}\rightarrow\mathbb{R}$ and $K:=\{x\in X\,|\,\|x\|\le r(\varphi(x))\}$. Then every map $f\in\mathcal{H}_\mathrm{G}(X\setminus K,Y)$ has a unique extension $\tilde f\in\mathcal{H}_\mathrm{G}(X,Y)$.
\item[(c)] Assume $X$ is an inner product space. Let two open balls $B_1,B_2\subset X$, such that $B_1\not\subset B_2$, $B_2\not\subset B_1$, and $B_1\cap B_2\ne\emptyset$. Set $\Omega:=\mathrm{co}(B_1\setminus\overline B_2)\subset X$. Then every map $f\in\mathcal{H}_\mathrm{G}(B_1\setminus\overline B_2,Y)$ has a unique extension $\tilde f\in\mathcal{H}_\mathrm{G}(\Omega,Y)$.
\item[(d)] Assume $X$ is an inner product space. Let $\varphi\in X^*\setminus\{0\}$ and two open balls $B\subsetneq\Omega\subset X$ centered at $0$. Then every map $f\in\mathcal{H}_\mathrm{G}\big(\Omega\setminus(\ker\varphi\setminus B),Y\big)$ has a unique extension $\tilde f\in\mathcal{H}_\mathrm{G}(\Omega,Y)$. On the other hand, for fixed $y\in Y\setminus\{0\}$, the map $\frac1\varphi\cdot y\in\mathcal{H}_\mathrm{G}(\Omega\setminus\ker\varphi,Y)$ has no extension from $\mathcal{H}_\mathrm{G}(\Omega,Y)$; this shows that the condition (ii) from Theorem~\ref{t.outer} cannot be dropped.
\end{description}
\end{example}

\noindent In each case we indicate how to use Theorem~\ref{t.outer} (Corollary~\ref{c.outer} does not apply).\\
(a). For $u\in X\setminus(\mathbb{C}\cdot e)$ and $c=-te$, with $t>0$ large enough.\\
(b). For $u\in\ker\varphi\setminus\{0\}$ and $c=te$, with $e\in\ker\varphi\setminus(\mathbb{C}\cdot u)$ and $t>0$ large enough.\\
(c). Let $a,b\in X$ denote the centers of the two balls. We may apply the theorem for $K:=\overline B_2\cap\Omega$ and $u\in\{b-a\}^\perp\setminus\{0\}$, and $c\in(B_1\setminus\overline B_2)\cap L_a(b-a)$.\\
(d). For $K:=\Omega\cap\ker\varphi\setminus B$ and $u\in(\ker\varphi)^\perp\setminus\{0\}$, and $c=0$.

As pointed out in the introduction, in Theorem~\ref{t.Hartogs} we may replace the compactness requirement on $K$ by significantly weaker assumptions.

\begin{corollary}[Hartogs extension]\label{c.Hartogs}
Let $n\ge2$, and the sets $K\subset\Omega\subset\mathbb{C}^n$, such that $\Omega$ and $\Omega\setminus K$ are open and $\Omega\setminus K$ is connected. Assume there exists $u\in\mathbb{C}^n\setminus\{0\}$, with the property that
\[K\cap L\mbox{ is compact, for every }L\in\Gamma_2(\mathbb{C}^n),\,L\parallel u.\]
Then every map $f\in\mathcal{H}(\Omega\setminus K)$ has a unique extension $\tilde f\in\mathcal{H}(\Omega)$. Furthermore, $\tilde f(\Omega)=f(\Omega\setminus K)$.
\end{corollary}

\begin{proof}
The conclusion is immediate, by Corollary~\ref{c.outer}.
\end{proof}

\begin{example}[Hartogs extension]\label{ex.Hartogs}
For every open set $\Omega\subset\mathbb{C}^3$, the subset
\[K:=\Omega\cap\{(z,z^2,0)\,|\,z\in\mathbb{C}\}\]
satisfies the condition from Corollary~\ref{c.Hartogs} with $u=(0,0,1)$, but $K$ may not be bounded or closed.
\begin{description}
\item[(a)] $K$ is bounded and not closed, and $K\cup(\mathbb{C}^n\setminus\Omega)$ is path-connected for $\Omega=B_{\mathbb{C}^3}(0,2)$.
\item[(b)] $K$ is unbounded and closed for $\Omega=\mathbb{C}^3$.
\item[(c)] $K$ is unbounded and not closed, and $K\cup(\mathbb{C}^n\setminus\Omega)$ is path-connected for $\Omega=\{(z_1,z_2,z_3)\in\mathbb{C}^3\,|\,\mathrm{Re}(z_1)>0\}$.
\end{description}
\end{example}

\bibliographystyle{amsplain}

\end{document}